\documentclass{article}
\usepackage[english]{babel}
\usepackage{amsfonts}
\usepackage{amsmath,amsthm,amssymb,latexsym}
\usepackage{color}
\usepackage{authblk}
\usepackage{graphicx}
\usepackage{tikz}
\usepackage{stmaryrd}
\usepackage{mathrsfs,upgreek}
\usepackage{eucal}
\usepackage{enumerate}
\usepackage{changes}

\newtheorem{theorem}{Theorem}[section]
\newtheorem{lemma}[theorem]{Lemma}
\newtheorem{proposition}[theorem]{Proposition}
\newtheorem{corollary}[theorem]{Corollary}
\theoremstyle{definition}

\newcommand{\op}[1]{\textrm{\upshape #1}}

\newcommand{\join}{\vee}

\newcommand{\back}{\backslash}
\newcommand{\meet}{\wedge}

\newcommand{\la}{\langle}
\newcommand{\ra}{\rangle}
\newcommand{\alg}[1]{{\textbf{\upshape #1}}}  %
\newcommand{\vv}[1]{\mathsf {#1}}

\renewcommand{\a}{\alpha}

\renewcommand{\d}{\delta}
\newcommand{\f}{\varphi}
\newcommand{\g}{\gamma}
\newcommand{\e}{\varepsilon}
\renewcommand{\th}{\theta}
\renewcommand{\o}{\omega}

\newcommand{\sse}{\subseteq}

\newcommand{\app}{\approx}
\newcommand{\HH}{{\mathbf H}}  
\newcommand{\II}{{\mathbf I}} 
\newcommand{\SU}{{\mathbf S}} 
\newcommand{\PP}{{\mathbf P}}   
\newcommand{\VV}{{\mathbf V}}   

\newcommand{\two}{\mathbf 2}

\newcommand{\itemb}{\item[$\bullet$]}

\newcommand{\OO}{{\mathbf O}}

\newcommand{\BL}{\vv B\vv L}
\newcommand{\BH}{\vv B\vv H}

\newcommand{\WH}{\vv W\vv H}

\newcommand{\ib}{\item[$\bullet$]}

\newcommand{\Con}[1]{\operatorname{Con}(\alg #1)}

\newcommand{\vuc}[2]{#1_1,\dots,#1_{#2}}

\newcommand{\imp}{\rightarrow}

\makeatletter
\def\square{\RIfM@\bgroup\else$\bgroup\aftergroup$\fi
  \vcenter{\hrule\hbox{\vrule\@height.6em\kern.6em\vrule}\hrule}\egroup}
\makeatother
\newcommand{\smlcirc}{\raise3pt\hbox{\textrm{\circle{3.3}}}}

\newcommand{\myfrac}[2]{\dfrac{#1}{\lower.5ex\hbox{$#2$}}}
\mathchardef\hu="0362

\renewcommand{\e}{\varepsilon}

\newcommand{\lr}{ {\slash}}
\newcommand{\rr}{ {\backslash}}

\begin{document}

\title{Strictly join irreducible varieties of residuated lattices}
\author{Paolo Aglian\`{o}\\
DIISM,\\
Universit\`a di Siena, Italy\\
agliano@live.com
\and
 Sara Ugolini\\
 Artificial Intelligence Research Institute,\\
 Spanish National Research Council, Barcelona, Spain\\
 sara@iiia.csic.es}
\date{}
\maketitle
\abstract{We study (strictly) join irreducible varieties in the lattice of subvarieties of residuated lattices. We explore the connections with well-connected algebras and suitable generalizations, focusing in particular on representable varieties. Moreover we find weakened notions of Halld\'en completeness that characterize join irreducibility. We characterize strictly join irreducible varieties of basic hoops, and use the generalized rotation construction to find strictly join irreducible varieties in subvarieties of $\mathsf{MTL}$-algebras. We also obtain some general results about linear varieties of residuated lattices, with a particular focus on representable varieties, and a characterization for linear varieties of basic hoops. }

\section{Introduction}
Substructural logics constitute a large class of logical systems algebraizable in the sense of Blok-Pigozzi, where the semantical characterization of provability of the Lindenbaum-Tarski algebraization extends to a characterization of logical deducibility via the algebraic equational consequence (see \cite{GJKO} for a detailed investigation). Substructural logics encompass classical logic, intuitionistic logic, fuzzy logics, relevance logics and many other systems, all seen as logical extensions of the Full Lambek calculus $\mathcal{FL}$. As a consequence of algebraizability, all extensions of $\mathcal{FL}$ are also algebraizable,  and the lattice of axiomatic extensions is dually isomorphic to the subvariety lattice of the algebraic semantics, given by the variety of $\mathsf{FL}$-algebras. In this work we are interested in the positive fragment of $\mathcal{FL}$ (the system obtained by removing the constant $0$, and consequently negation, from the language), $\mathcal{FL}^{+}$, whose corresponding algebraic semantics is given by the variety of residuated lattices $\mathsf{RL}$.

Our investigation will be carried on in the algebraic framework, and goes in the direction of gaining a better understanding of the lattice of subvarieties of residuated lattices (thus, equivalently, the lattice of axiomatic extensions of the corresponding logics). In particular we study properties, and in some relevant cases we find characterizations, of those varieties that in the lattice of subvarieties are join irreducible or strictly join irreducible.
Kihara and Ono showed that in the presence of integrality and commutativity join irreducibility of a variety is characterized by both a logical property, Halld\'en completeness, and by an algebraic property of the generating algebras, well-connectedness (\cite{KiharaOno2008}, anticipated in \cite{GJKO}).

We show how both these notions can be generalized for non-integral, non-commutative subvarieties of $\mathsf{RL}$, characterizing join irreducibility in a large class of residuated lattices, that include for instance all normal varieties, representable varieties, and $\ell$-groups. Moreover, we answer two questions left open in \cite{KiharaOno2008} in our more general setting. In particular, we show that in our more general framework being join irreducible is equivalent to being generated by a subdirectly irreducible algebra, and that in general not all subdirectly irreducible algebras generate strictly join irreducible varieties.
A key role is played by results implicit in \cite{Galatos2003} concerning the axiomatization of the join of varieties of residuated lattices.

We then focus on representable varieties, where well-connectedness is shown to correspond to being totally ordered. In the further presence of divisibility, we use the ordinal sum construction to find sufficient conditions for strict join irreducibility. Then we focus on the variety of basic hoops, where a characterization of strictly join irreducible varieties is given. The representation shows to be analogous to the one found in \cite{AguzzoliBianchi2020} for $\mathsf{BL}$-algebras (of which basic hoops are $0$-free subreducts) and is constructed based off the representation of totally ordered hoops with finite index given in \cite{AglianoMontagna2017}. Given the results in \cite{AglianoUgolini2019a}, using the generalized rotation construction in \cite{BMU2018}, we are able to lift the results found about the lattice of subvarieties of basic hoops to particular intervals of the lattice of subvarieties of $\mathsf{MTL}$-algebras (representable, commutative, integral $\mathsf{FL}$-algebras). In the end we focus on particular join irreducible varieties, namely linear varieties, finding some general characterizations in the presence of representability and a characterization for linear varieties of basic hoops, that again we lift to particular subvarieties of $\mathsf{MTL}$.

\section{Preliminaries}

A {\bf residuated lattice} is an algebra  $\alg A = \la A,\join,\meet,\cdot,\lr,\rr, 1\ra$ where
\begin{enumerate}
\item $\la A, \join, \meet\ra $ is a lattice;
\item $\la A, \cdot,1\ra$ is a monoid;
\item $\lr$ and $\rr$ are the right and left divisions w.r.t. $\cdot$, i.e. $x \cdot y \leq z$ iff $y \leq x \rr z$ iff $x \leq z \lr y$.
\end{enumerate}
Residuated lattices form a variety $\mathsf{RL}$ and an axiomatization, together with the many equations holding in these very rich structures, can be found in \cite{BlountTsinakis2003}.

A residuated lattice $\alg A$ is {\bf integral} if it satisfies the equation $x \le 1$; it  is {\bf commutative} if $\cdot$ is commutative;
it is {\bf divisible} if  the ordering of $\alg A$ is the {\em inverse divisibility ordering} i.e. for $a,b \in A$
$$
a\le b \qquad\text{if and only if}\qquad\text{there are $c,d \in A$ with $a=cb=bd$.}
$$
The first two properties are clearly equational, while the third is equivalent to the equations (\cite{JipsenTsinakis2002})
$$
(x\lr y)y \app x \meet y \app y(y\rr x).
$$
Finally a residuated lattice is {\bf cancellative} if the underlying monoid is cancellative in the usual sense; it turns out that this property is equational as well \cite{BahlsColeGalatos2003}.
A residuated lattice is cancellative if and only if it satisfies
$$
xy \lr y \app x \qquad  x \rr xy \app y.
$$
So the classes of residuated lattices that satisfy any combination of integrality, commutativity, divisibility  or cancellativity are subvarieties of $\mathsf{RL}$. We shall call the variety of integral residuated lattices $\mathsf{IRL}$, commutative residuated lattices $\mathsf{CRL}$, and their intersection $\mathsf{CIRL}$.

Residuated lattices with an extra constant $0$ in the language are called $\mathsf{FL}$-algebras, since they are the equivalent algebraic semantics of the Full Lambek calculus $\mathcal{FL}$ (for precise definition of this calculus, or more details on algebraizability, see \cite{GJKO}). Residuated lattices are then the equivalent algebraic semantics of the positive (i.e., $0$-free) fragment of $\mathcal{FL}$, $\mathcal{FL}^{+}$.

In particular, for every variety of residuated lattices $\vv V$ over a language $\mathscr{L}$, its corresponding logic is $\mathcal L_\vv V = \{ \varphi \in Fm_{\mathscr{L}}: \vv V \models \varphi \geq 1\}$, where $Fm_{\mathscr{L}}$ is the set of formulas over $\mathscr{L}$. Conversely, given any extension $\mathcal{L}$ of $\mathcal{FL}^{+}$ axiomatized by the set of formulas $\Phi$, its equivalent algebraic semantics is the subvariety $\vv V_\mathcal L$ of $\mathsf{FL}$ axiomatized by the set of equations $\{\varphi \geq 1 : \varphi \in \Phi\}$.
Moreover, the following is shown in \cite{GalatosOno06}.
\begin{theorem}
The maps $\mathbf{L}: \vv V \mapsto \mathcal L_\vv V$ and $\mathbf{V}: \mathcal{L} \mapsto \vv V_{\mathcal{L}}$ are mutually inverse, dual lattice isomorphisms between the lattice of subvarieties of $\mathsf{RL}$ and the lattice of logical extensions of $\mathcal{FL}^{+}$.
\end{theorem}

It is useful to observe that a conjunction of a finite number of equations is equivalent in $\mathsf{RL}$ to a single inequality of the kind $p \geq 1$, or equivalently to an equation of the kind $p\land 1 \approx 1$ (Lemma 3.1 in \cite{Galatos2004}).

With respect to their structure theory, residuated lattices are congruence permutable and congruence point-regular at 1, and also congruence distributive (since they have lattice terms). It follows that $\mathsf{FL}$-algebras and residuated lattices are {\em ideal determined} \cite{GummUrsini1984} and have a good theory of ideals in the sense of \cite{AglianoUrsini1992}. The role of the ideals is played by the {\bf congruence filters}; we call a {\bf filter} $F$ of $\alg A$ a subset of $A$ that is a lattice filter containing $1$ and is closed under multiplication.
Both filters an congruence filters form algebraic lattices; if $A^+ = \{a: a \ge 1\}$, then for any $\th \in \Con A$,
$$
A^+/\th = \bigcup\{a/\th: a \ge 1\}
$$
is a congruence filter. Moreover the two mappings
$$
\th \longmapsto A^+/\th \qquad\qquad F \longmapsto \th_F = \{(a,b): a\lr b, b\lr a \in F\}
$$
are mutually inverse order preserving maps from  $\Con A$ to the congruence filter lattice of $\alg A$, that are therefore isomorphic. Notice that $A^+/\th$ defines the same set as $\uparrow\![1]_{\theta}$, used for example in \cite{GJKO}.

A residuated lattice is {\bf normal} if every filter is a congruence filter and a variety of residuated lattices is {\bf normal} if each of its members is normal.
To determine which varieties are normal a description of the congruence filter generated by a subset is in order. Following \cite{JipsenTsinakis2002} we define $l_a(x) = a \rr xa \meet 1$ and $r_a(x) = ax\lr a \meet 1$ and we call them the {\bf left and right conjugates} of $x$ with respect to $a$.

\begin{lemma}\label{filcong} \cite{JipsenTsinakis2002} Let $\alg A$ be a residuated lattice; then a filter $F$ of $\alg A$ is a congruence filter if and only if for all $a \in F$ and $b \in A$,
$l_b(a), r_b(a) \in F$.
\end{lemma}

This lemma allows a handy description of the congruence filter generated by a subset $X$ of a residuated lattice $\alg A$; an {\bf iterated conjugate} in $\alg A$ is a unary term $\g_{a_1}(\g_{a_2}(\dots \g_{a_n}(x)))$ where
$\vuc an \in A$ and $\g_{a_i} \in \{l_{a_i},r_{a_i}\}$ for $i = 1 \ldots n, n \in \mathbb{N}$.
We will denote with $\Gamma^n(\alg A)$ the set of iterated conjugates in $\alg A$ of \emph{length n} (i.e. a composition of $n$ left and right conjugates).

\begin{corollary}\label{prinfil} Let $\alg A$ be a residuated lattice and let $X \sse A$; then the congruence filter generated by $X$ in $\alg A$ is the set
\begin{equation*}
\begin{aligned}
{\bf F}(X) = \{b \in A :\; &\g_1(a_1)\dots\g_n(a_n) \le b, \;n \in \mathbb N,\  \vuc an \in X\cup\{1\}, \\ &\g_i \in \Gamma^{k}(\alg A) \mbox{ for some } k \in \mathbb{N}, i = 1 \ldots n\}.
\end{aligned}
\end{equation*}
\end{corollary}

 Section 4 of \cite{GalatosOlsonRaftery2008} contains an implicit characterization of  normal residuated semilattices. Let's make it explicit for residuated lattices.

\begin{lemma}\label{normalS} For a residuated lattice $\alg A$ the following are equivalent:
\begin{enumerate}
\item $\alg A$ is normal;
\item every principal filter of $\alg A$ is a congruence filter;
\item for all $a\in A$, given any $b \in A$ there exist  $n,m \in \mathbb N$ (possibly depending on $b$) such that $(a\meet 1)^nb \le ba$ and $b(a \meet 1)^m \le ab$.
\end{enumerate}
\end{lemma}
\begin{proof}
(1) clearly implies (2). Assume then (2), and let $a \in A$; then the principal filter $F$ generated by $a$ is a congruence filter. If $a \geq 1$ (3) can be easily checked. From Lemma \ref{filcong} it follows that for any $b \in A$ $l_b(a), r_b(a) \in F$. This implies that there  are $n,m \in \mathbb N$ such that $a^m \le l_b(a)$ and $a^n \le r_b(a)$, from which (3) follows by residuation and order preservation. Finally assume (3) and let $F$ be a filter of $\alg A$; if $a \in F$ and $b \in A$ then there is an $n \in \mathbb N$ with $(a\meet 1)^n b \le ba$. It follows that $(a \meet 1)^n \le ba \lr b$  and  so $ r_b(a) \in F$ and, by similar argument, $l_b(a) \in F$ as well. By Lemma \ref{filcong} $F$ is a congruence filter and (1) holds.
\end{proof}

Of course there might be no bound on $n$ in the above lemma; but if there is one (for instance, if $\alg A$ is finite), then $\VV(\alg A)$ is a normal variety of residuated lattices.  On the other hand if an entire variety is normal then we can say more.
Let $\mathsf{RL}^{\hat n}$ be the variety of residuated lattices satisfying:
\begin{align*}
(x \meet 1)^n y \le y x\\
y(x\meet 1)^n \le xy.
\end{align*}

\begin{lemma}\label{lemma:normal} A variety $\vv V$ of residuated lattices is normal if and only if $\vv V \sse \mathsf{RL}^{\hat n}$ for some $n$.
\end{lemma}
\begin{proof} One direction is obvious. Assume then that $\vv V$ is normal and let $\alg F_\vv V(x,y)$ be the free 2-generated algebra in $\vv V$. Then $\alg F_\vv V(x,y)$ is normal and by Lemma \ref{normalS} there are
$m,k$ with  $(x \meet 1)^m y\le yx$ and $y(x \meet 1)^k \le xy$. Take $n = \max\{m,k\}$; then in $\alg F_\vv V(x,y)$
$$
(x \meet 1)^n y\le yx\qquad y(x \meet 1)^n \le xy.
$$
But since $\alg F_\vv V(x,y)$ is free, these inequalities hold in the entire $\vv V$, i.e. $\vv V \sse \mathsf{FL}^{\hat n}$.
\end{proof}

Thus in particular, every variety of commutative residuated lattices is normal.

\section{Join irreducible varieties and iterated conjugates}

Given any variety $\vv V$ we can consider its lattice of subvarieties $\Lambda (\vv V)$; we will say that a subvariety $\vv W \sse \vv V$ is {\bf join irreducible} ({\bf strongly join irreducible}) if $\vv W$ is a  join irreducible (strongly join irreducible) member of $\Lambda (\vv V)$.
We will see how in relevant cases  (strong) join irreducibility can be characterized by both a property of the corresponding logic and an algebraic description of the generating algebras.

In particular, we start by discussing the concept of well-connected algebra, that was introduced by L. Maksimova \cite{Maksimova1986} to characterize the disjunction property in intermediate logics, i.e. those logics $\mathcal L$ for which $\vv V_\mathcal L$ is a variety of Heyting algebras.  A substructural logic $\mathcal L$ has the {\bf disjunction
property} if whenever $\f \join \psi$ is a theorem of $\mathcal L$, in symbols $\mathcal L \vdash \f \join \psi$, then either $\mathcal L \vdash \f$ or $\mathcal L\vdash \psi$. Likewise a Heyting algebra $\alg A$ is {\bf well-connected} if $1$ is join irreducible, i.e. $a \join b =1$ implies either $a=1$ or $b=1$.  A weaker property is Halld\'en completeness; a logic $\vv L$ is {\bf Halld\'en complete} if it has the disjunction property w.r.t. to any pair of formulas that have no variables in common. Classical logic is Halld\'en complete but does not have the disjunction property, thus differentiating the two concepts.
As shown in \cite{KiharaOno2008} these concepts are  connected in commutative integral residuated lattices.

\begin{theorem} \label{KO}(Theorem 2.5 in \cite{KiharaOno2008}) For a variety $\vv V$ of commutative and integral residuated lattices the following are equivalent:
\begin{enumerate}
\item $\mathcal L_\vv V$ is Halld\'en complete;
\item  $\vv V$ is join irreducible;
\item  $\vv V = \vv V(\alg A)$ for some well-connected algebra $\alg A$.
\end{enumerate}
\end{theorem}

How can we extend the definition of well-connected to the nonintegral case?  The solution proposed in  \cite{KiharaOno2008} (and later followed in \cite{HorcikTerui2011}) is to define a residuated lattice $\alg A$  to be {\bf well-connected} if $1$ is {\bf join prime} in $\alg A$, i.e. $a \join b \ge 1$ implies $a \ge 1$ or $b \ge 1$.

We observe straight away that  neither integrality nor commutativity are needed to prove  that
(3) implies (2).

\begin{lemma}\label{KOlemma} Let $\vv V$ be a variety of residuated lattices;  if  $\vv V = \VV(\alg A)$ for some well-connected algebra $\alg A \in \vv V$ then $\vv V$ is join irreducible.
\end{lemma}
\begin{proof}  Let $\vv V = \VV(\alg A)$ for some well-connected algebra $\alg A \in \vv V$, and suppose by way of contradiction that $\vv V$ is not join irreducible, i.e. $\vv V = \vv W  \join \vv Z$ for some proper subvarieties $\vv W$ and $\vv Z$. Then $\vv W$ and $\vv Z$ must be incomparable,
hence there is an equation $p(\vuc xn) \ge 1$ holding in $\vv W$ but not in $\vv Z$ and an equation $q(\vuc ym) \ge 1$ holding in $\vv Z$ but not in $\vv W$. Clearly neither equation can hold in
$\vv V$, so there are $\vuc an,\vuc bm \in A$ with $p(\vuc an) \not \ge 1$ and $q(\vuc bm) \not\ge 1$; since $\alg A$ is well-connected we must have $p(\vuc an) \join q(\vuc bm) \not\ge 1$ hence
$\vv V = \VV(\alg A) \not\vDash  p(\vuc xn) \join q(\vuc ym) \ge 1$. On the other hand clearly $\vv W,\vv Z \vDash  p(\vuc xn) \join q(\vuc ym) \ge 1$ and since $\vv V = \vv W \join \vv Z$ we must have
$\vv V \vDash p(\vuc xn) \join q(\vuc ym) \ge 1$. This is a contradiction, derived from the assumption that $\vv V$ was not join irreducible.
\end{proof}

The other implications in the general case however do not hold; an analysis of  the Kihara-Ono construction reveals at once that there are two critical points. If $\vv V$ is a variety of commutative residuated lattices then:
\begin{enumerate}
\ib every subdirectly irreducible algebra in $\vv V$  is well-connected (\cite{KiharaOno2008}, Lemma 2.2);
\ib  if $\vv W, \vv Z$ are subvarieties of $\vv V$ axiomatized (relative to $\vv V$) by $p \ge 1$ and $q \ge 1$ (and we make sure that $p$ and $q$ have no variables in common), then
$\vv W \join \vv Z$ is axiomatized relative to $\vv V$ by $ p \join q \ge 1$ (\cite{KiharaOno2008}, Lemma 2.1).
\end{enumerate}
Both statements are false if we remove commutativity; for the first it is easy to find a finite and integral residuated lattice that is simple but not well-connected (for instance the example below Lemma 3.60 in \cite{GJKO}),
while the second fails fore more general reasons discussed at length in \cite{Galatos2004}.

In what follows we will describe classes of varieties of residuated lattices for which an analogous of Theorem \ref{KO} can be proved. To do so we will adapt to our purpose  part of the theory developed in \cite{Galatos2004} about satisfaction of formulas generated by iterated conjugates.

We define a set $B^n(x,y)$ of equations  in two variables $x,y$  for all $n \in \mathbb N$ in the following way; let $\Gamma^n$ be the set of iterated conjugates of \emph{length n} (i.e. a composition of $n$ left and right conjugates) over the appropriate language, with $\Gamma^{0} = \{l_{1}\}$ (for a more general definition, here not needed, see \cite{Galatos2004}, page 229). For all $n \in \mathbb{N}$
$$
B^{n}(x,y) = \{ \gamma_{1}(x) \join \gamma_{2}(y) \app  1 \,:\, \gamma_{1}, \gamma_{2} \in \Gamma^{n} \}.
$$
Let $\alg A$ be  a residuated lattice and $a,b \in A$; we say that $\alg A$ {\bf satisfies} $B^n(a,b)$, in symbols $\alg A \vDash B^n(a,b)$ if $\alg A,a,b \vDash B^n(x,y)$. i.e. $\g_1(a) \join \g_2(b) = 1$ for all $\g_1,\g_2 \in \Gamma^n(\alg A)$. We say that $\alg A$ satisfies ($G_{n,k}$) if for all $a,b \in A$, if $\alg A \vDash B^n(a,b)$, then $\alg A\vDash B^k(a,b)$.

\begin{lemma}\label{Glemma}  Let $\alg A$ be a residuated lattice.
\begin{enumerate}
\item for all $n\in \mathbb N$, for all $a,b \in A$, if $\alg A\vDash B^n(a,b)$ then $\alg A\vDash B^{n-1}(a,b)$ ;
\item for all $n \in \mathbb N$ if $\alg A$ satisfies  ($G_{n,n+1}$) then it satisfies ($G_{n,k}$) for all $k \ge n$.
\end{enumerate}
\end{lemma}
\begin{proof} The proof of (1) is straightforward, once we have observed that $l_{1}(a) = a \land 1$ and all conjugates are smaller than $1$.

For (2), suppose that $\alg A$ satisfies ($G_{n,n+1}$),
we will show that $\alg A$ satisfies $(G_{n+1,n+2})$. As the relation $\imp$ is clearly transitive this implies that $(G_{n,n+2})$ and hence, by induction, ($G_{n,k}$) for $k \ge n$.

Let then $a,b \in A$ such that $\alg A\vDash B^{n+1}(a,b)$. Given any $\hat{\g}_1,\hat{\g}_2 \in \Gamma^{n+2}$, there are $\d_1,\d_2,\d'_1,\d'_2 \in \Gamma^1$, $\g_1,\g_2 \in \Gamma^n$ such that
$$
\hat{\g}_1(x) = \d_1\g_1\d_2(x) \qquad\quad \hat{\g}_2(y) = \d'_1\g_2\d'_2 (y).
$$
Since $\alg A\vDash B^{n+1}(a,b)$, then  $\g_1\d_2(a) \join \g_2\d'_2(b) = 1$ for all $\g_1,\g_2 \in \Gamma^n(\alg A)$.  If $u=\d_2(a)$ and $v =\d'_2(b)$, then $\alg A$ satisfies
$B^n(u,v)$; but since $\alg A$ satisfies ($G_{n, n+1}$), then it satisfies $B^{n+1}(u,v)$ as well. This implies
$$
\hat{\g}_1(a) \join \hat{\g}_2(b) = \d_1\g_1(u) \d'_1\g_2(v) = 1,
$$
which in turn yields that $\alg A\vDash B^{n+2}(a,b)$, and then $\alg A$ satisfies $(G_{n+1,n+2})$. Thus, by transitivity and induction, it satisfies ($G_{n,k}$) for $k \ge n$.
\end{proof}

Let's look closely at the condition ($G_{0,1}$); first we observe the following.

\begin{lemma} Let $\vv V$ be a variety of residuated lattices;  $\vv V$ satisfies ($G_{0,1}$) (i.e.  $\alg A$ satisfies ($G_{0,1}$) for all $\alg A \in \vv V$)  if and only if it satisfies the quasi equation
\begin{equation}
x \join y \app 1 \qquad \Longrightarrow\qquad l_w(x) \join r_z(y) \app 1 \tag{$G$}.
\end{equation}
\end{lemma}
\begin{proof} One direction is  easy to check. Indeed, suppose $\vv V$ satisfies ($G_{0,1}$), and let $\alg A \in \vv V, a, b \in A $ such that $a \lor b = 1$. Then $a = a \land 1, b = b \land 1$, and ($G$) directly follows from ($G_{0,1}$), that corresponds to the set of quasi equations: $$(x \land 1) \lor (y \land 1) \app 1  \qquad \Longrightarrow\qquad \g_{1}(x) \join \g_{2}(y) \app 1, \mbox{ for } \g_{1}, \g_{2} \in \Gamma^{1}.$$

For the other, suppose ($G$) holds, and let $\alg A \in \vv V$ and $a,b \in A$; then $\alg A \vDash B^0(a,b)$ if and only if $(a \meet 1) \join (b \meet 1) =1$.  Then by ($G$)  for all $u,v \in A$
$$
1 = l_u(a\meet 1) \join r_v(b \meet 1) \le l_u(a) \join r_v(b) \le 1.
$$
Thus $l_u(a) \join r_v(b) = 1$, and in particular for all $u \in A$, $l_u(a) \join r_{1}(b) = l_u(a) \join (b\meet 1) = 1$ and so, applying again the same reasoning since $l_{u}(a) \leq 1$, for all $v \in A$
$$
1= l_u(a) \join l_v(b\meet 1) \le l_u(a) \join l_v(b) \le 1.
$$
The proof that $r_u(a) \join r_v(b) =1$ is similar, so we conclude that $\alg A \vDash B^1(a,b)$; therefore $\alg A$ satisfies ($G_0$) and so does $\vv V$.
\end{proof}

($G$) is trivially satisfied by any commutative variety of residuated lattices, but moreover:

\begin{lemma}
Any normal variety $\vv V$ of residuated lattices satisfies ($G$).
\end{lemma}
\begin{proof}
Let $\alg A \in \vv V$, and consider $x, y \in A$ with $x \lor y = 1$, thus clearly $x, y \leq 1$. Since both conjugates are below $1$, $l_w(x) \join r_z(y) \leq 1$. Thus we need to show that $1 \leq l_w(x) \join r_z(y) $.

By Lemma \ref{lemma:normal}, $(y \meet 1)^n z \le zy$ and $ w(x\meet 1)^n \le xw$, thus $(x \land 1)^{n} \leq l_w(x)$ and $(y \land 1)^{n} \leq r_{z}(y)$. From \cite{Galatos2004}, Lemma 3.20, we get that if $x \lor y = 1$ then $x^{n} \lor y^{n} = (x \land 1)^{n} \lor (y \land 1)^{n} =  1$. Thus, $$1 \leq  (x \land 1)^{n} \lor (y \land 1)^{n} \leq l_w(x) \join r_z(y)$$ and the proof is completed.
\end{proof}

There are also non normal varieties satisfying $(G)$, for instance, {\em representable} varieties (see Section \ref{repr} below).
We observe also that the variety of $\ell$-groups satisfies ($G_{1,2}$) (see \cite{Galatos2004}, p. 235).
The last lemma we need is implicit in \cite{Galatos2004}.

\begin{lemma}\label{galatos}  Let $\vv V$ be a variety of residuated lattices and let $p(\vuc xn) \ge 1$, $q(\vuc ym) \ge 1$ be two inequalities  not holding in $\vv V$. If $\vv W$ and $\vv Z$ are the subvarieties axiomatized by $p \meet 1 \app 1$ and $q \meet 1 \app 1$ respectively, then $\vv W \join \vv Z$ is axiomatized by the set  $B(p,q) = \bigcup_{n\in \mathbb N} B^n(p,q)$.  Moreover if $\vv V$ satisfies ($G_{l,l+1}$) for some $l \in \mathbb{N}$ then  $\vv W \join \vv Z$ is axiomatized by the finite set $B^l(p,q)$.
\end{lemma}
\begin{proof}

The first part is a consequence of Corollary 4.3 in \cite{Galatos2004}. The second part follows from Theorem 4.4 (1) in \cite{Galatos2004}, where it is shown that $\vv W \join \vv Z$ is axiomatized by $B^l(p,q)$ plus the (finite set of) equations implying ($G_{l,l+1}$), which here can be omitted since ($G_{l,l+1}$) holds in $\vv V$.
\end{proof}

\section{$\Gamma$-connectedness and $\Gamma$-completeness}

The {\bf negative cone} of a residuated lattice $\alg A$  is an integral residuated lattice $\alg A^-$ whose universe is $A^- = \{a \in A: a \le 1\}$ and the operations are defined as follows: the lattice operations and
the product are the same as in $\alg A$ but the residuations are defined as $a \lr_1 b := a \lr b\meet 1$ and $a \rr_1 b := a \rr b \meet 1$.  The connections between $\alg A^-$ and $\alg A$ are strict; for instance it can be shown that $\Con A$ and $\Con A^-$ are isomorphic \cite{GJKO}.
A residuated lattice $\alg A$ is {\bf weakly well-connected} if $\alg A^-$  is well-connected; we have the following obvious lemma.

\begin{lemma} For a residuated lattice $\alg A$ the following are equivalent:
\begin{enumerate}
\item $\alg A$ is weakly well-connected;
\item $1$ is join irreducible in $\alg A$;
\item for $a,b \in A$ if $(a \meet 1) \join (b \meet 1) =1$, then either $a \ge 1$ or $b \ge 1$.
\end{enumerate}
\end{lemma}

Clearly if $\alg A$ is well connected, then it is  weakly well-connected and the two concepts coincide in varieties that are $1$-distributive, i.e. satisfy the equation
$$
(x \join y) \meet 1 \le (x \meet 1) \join (y \meet 1).
$$

We want to generalize these concepts;
we say that a residuated lattice $\alg A$ is {\bf ${\bf\Gamma^n}$-connected} if for all $a,b \in A$,  if $\g_1(a) \join \g_2(b) =1$ for all $\g_1,\g_2 \in \Gamma_n(\alg A)$, then either $a \ge 1$ or $b \ge 1$.
To prove the next result we need a lemma.

\begin{lemma}\label{product} Let $\alg A$ be a residuated lattice and let $\vuc am,\vuc bl \in A$; if $a_i \join b_j =1$ for all $i\le l$ and $j \le m$, then
for all $r,s$ and $\vuc dr \sse \{\vuc al\}$ and $\vuc es \sse \{\vuc bm\}$ we have
$$
e_1 \dots e_s \join d_1\dots d_r =1.
$$
\end{lemma}

The proof is a simple finite induction using the equations holding in residuated lattices (see Lemma 3.2 in \cite{Galatos2004}).

\begin{lemma}\label{wellconnected} Let $\vv V$ be a variety of residuated lattices that satisfies ($G_{n,n+1}$).
Then every subdirectly irreducible algebra in $\vv V$ is $\Gamma^n$-connected.
\end{lemma}
\begin{proof} Let $\alg A$ be subdirectly irreducible and let $a,b \in A$ with $\g_1(a) \join \g_2(b) =1$ for all $\g_1,\g_2 \in \Gamma^n(\alg A)$, we will show that then either $a \ge 1$ or $b \ge 1$. Pick any $\g_1,\g_2 \in \Gamma^n(\alg A)$, then if $d \in \op{Fil}_\alg A(\g_{1}(a)) \cap \op{Fil}_\alg A(\g_{2}(b))$ by using the description of the filter generated by
an element in Corollary \ref{prinfil} we get that
$$
\d_1(\g_1(a)) \dots\d_l(\g_1(a))  \join \d'_1(\g_2(b)) \dots \d'_m(\g_2(b)) \le d
$$
for some $\d_i,\d'_j \in \Gamma^k$, for $i = 1 \ldots l, j = 1 \ldots m$ and $k \in \mathbb{N}$. Indeed the iterated conjugates can all be considered of the same (maximum) length $k$.  Let $\e_i = \d_i\g_1$ and $\e'_j = \d'_j\g_2$,  then $\e_i,\e_j \in \Gamma^{n+k}$ for $i = 1 \ldots l, j = 1 \ldots m$. As $\alg A$ satisfies $(G_{n,n+1})$ by Lemma \ref{Glemma} it satisfies $(G_{n,n+k})$ and then
$$
\e_i(a) \join \e'_j(b) = 1\qquad\text{for all $i,j$};
$$
now we apply Lemma \ref{product}, and we get
$$
\e_1(a)\dots \e_l(a) \join \e'_1(b)\dots \e'_m(b) = 1
$$
so $d \ge 1$.
Since $d$ was generic we must have $\op{Fil}_\alg A (\g_1(a)) \cap \op{Fil}_\alg A(\g_2(b)) = A^+$;  but since $\alg A$ is subdirectly irreducible the isomorphism between the congruence lattice and the filter lattice of $\alg A$
forces either $\op{Fil}_\alg A(\g_1(a)) = A^+$ or $\op{Fil}_\alg A(\g_2(b)) = A^+$. This works for any choice of $\g_1,\g_2 \in \Gamma^n$, thus in particular it works for $\g_{1} = \g_{2} = l_{1} \circ \ldots \circ l_{1} = l_{1}$, thus we get that either $a \land 1 \ge 1$ or $b \land 1 \geq 1$, that is, either $a \ge 1$ or $b \ge 1$ as desired. Thus $\alg A$ is $\Gamma^n$-connected.
\end{proof}

Finally let's   complete the connection with logic. Let $\mathcal{L}$ be a substructural logic over $\mathcal{FL}^{+}$; given any two axiomatic extensions $\mathcal{L}_{1}$ and $\mathcal{L}_{2}$ axiomatized by formulas $\phi$ and $\psi$ respectively, for any $n$ Lemma \ref{galatos} implicitly gives a set of formulas $B^n_\mathcal L(\phi, \psi)$ such that $B_\mathcal L(\phi,\psi) = \bigcup_{n\in \mathbb N} B^n_\mathcal L(\phi,\psi)$  axiomatizes the intersection $\mathcal{L}_{1} \cap \mathcal{L}_{2}$, corresponding to the join of the varieties $\vv V_{\mathcal{L}_{1}} \lor \vv V_{\mathcal{L}_{2}}$.
We say that $\mathcal{L}$ is {\bf ${\bf \Gamma^n}$-complete} if for all formulas $\phi$ and $\psi$ which have no variables in common, if $\mathcal{L} \vdash B^n_\mathcal L(\phi, \psi)$  then either $\mathcal{L}\vdash \phi$ or $\mathcal{L} \vdash \psi$.

\begin{theorem} \label{mainwell} Let $\vv V$ be a variety of residuated lattices satisfying ($G_{n,n+1}$) for some $n \in \mathbb{N}$; then the following are equivalent.
\begin{enumerate}
\item $\mathcal L_\vv V$ is $\Gamma^n$-complete;
\item $\vv V$ is join irreducible;
\item $\vv V= \VV(\alg A)$ for some $\Gamma^n$-connected algebra $\alg A$.
\end{enumerate}
\end{theorem}
 \begin{proof}

 We show first that (1) implies (2). Assume that (2) fails, i.e. $\vv V = \vv W \join \vv Z$ and $\vv V \ne \vv W, \vv Z$. Let $p \ge 1$ be an equation holding in $\vv W$ but not in $\vv Z$ and $q\ge 1$ an equation holding in $\vv Z$ but not in $\vv W$, respectively; if $\vv W',\vv Z'$ are the varieties axiomatized relative to $\vv V$ by $p \ge 1, q \ge 1$ respectively, we have
$$
\vv V = \vv W \join \vv Z \sse \vv W' \join \vv Z' \sse \vv V;
$$
so we may assume that $\vv W,\vv Z$ are axiomatized relative to $\vv V$ by $p \ge 1$ and $q \ge 1$. Let $\f:= p \ge 1$ and $\psi := q \ge 1$; then $\mathcal{L}_\vv W \cap \mathcal{L}_\vv Z =\mathcal{L}_\vv V$ is axiomatized
by $B_{\mathcal L_\vv v}(\f,\psi)$. Hence $\mathcal L_\vv V \vdash B_{\mathcal L_\vv v}(\phi,\psi)$ but neither $\mathcal L_\vv V \vdash \f$ nor $\mathcal L_\vv V \vdash\psi$. Hence $\mathcal L_\vv V $ does not have the  $\Gamma^n$-completeness property.

We now show that (2) implies (1). Assume that (1) fails; hence there exist formulas $\phi$ and $\psi$ with no common variables such that $\mathcal{L}_{\vv V} \vdash B^n_\mathcal L(\phi, \psi)$  but neither $\mathcal{L}_{\vv V}\vdash \phi$ nor $\mathcal{L}_{\vv V} \vdash \psi$.
Then we can consider the two axiomatic extensions of $\mathcal{L}_{\vv V}$, $\mathcal{L}_{1}$ and $\mathcal{L}_{2}$, axiomatized by $\phi$ and $\psi$ respectively. This corresponds to taking two subvarieties of $\vv V$, $\vv W$ and $\vv Z$, defined by identities $p \meet 1\approx 1$ and $q\meet 1\approx 1$ respectively. Since $\vv V$ satisfies ($G_{n,n+1}$), by Lemma \ref{galatos}, $\vv W \lor \vv Z$ is finitely axiomatizable relative to $\vv V$  by $B^n(p,q)$ and the corresponding formula axiomatizing $\mathcal{L}_{1} \cap \mathcal{L}_{2}$ is $B^n_{\mathcal L_{\vv V}}(\phi, \psi)$. But we are assuming that $\mathcal{L}_{\vv V}\vdash B^n_\mathcal{L}(\phi, \psi)$, which implies that $\vv W \lor \vv Z = \vv V$;  thus $\vv V$ is not join irreducible, a contradiction. Hence (2) implies (1), and (1) and (2) are equivalent.

We now show the equivalence of (2) and (3).  Assume (2) and let $\alg A$ be any algebra such that $\vv V = \VV(\alg A)$; let $\{\th_i: \; i\in I\}$ be the set of all congruences of $\alg A$ such that $\alg A/\th_i$ is subdirectly irreducible. Next for any
equation $\e$ in the language of $\vv V$ we define $L_\e = \{i \in I: \alg A/\th_i \not\vDash \e \}$; we claim that $\Delta = \{L_\e: \vv V \not \vDash \e\}$ is a collection of nonempty subsets with the finite intersection property.

First note that if $\vv V \not\vDash \e$ then $\alg A\not\vDash \e$ and hence (by Birkhoff's Theorem) $\alg A/\th_i \not\vDash \e $ for some $i$; therefore  $L_\e \ne \emptyset$.  Next we prove that for any $\e$, $\delta$ such that $\vv V \not\vDash \e,\d$ there is a $\g$ such that $\vv V \not\vDash \g$ and  $L_\g \sse L_\e \cap L_\d$.

Let $\vv W, \vv Z$ be the subvarieties of $\vv V$ axiomatized by $\e$ and $\d$ respectively; since $\vv V \not\vDash \e,\d$ they are both proper subvarieties of $\vv V$. Let $\e$ be $p(\vuc xn) \ge 1$ and
 $\d$ be $q(\vuc ym) \ge 1$; since $\vv V$ satisfies ($G_{n,n+1}$), by Lemma \ref{galatos} $\vv W \join \vv Z$ is axiomatized by the finite set of equations $B^{n}(p,q)$, which contains in particular the equation $(p\meet 1) \join (q \meet 1) \app 1$.

Let $\g$ be the inequality equivalent to the conjunction of the equations in $B^{n}(p,q)$; if $i \notin L_\e$ then $\alg A/\th_i \vDash p \geq 1$, and thus also $\alg A/\th_i \vDash \g_{1}(p) \geq 1$ for any $\g_{1} \in \Gamma^{n}(\alg A/\th_i)$. Hence $\alg A_i/\th \vDash \g_{1}(p) \join \g_{2}(q) \app 1$ for all $\g_{1}, \g_{2} \in \Gamma^{n}$, and those are exactly the equations in $B^{n}(p,q)$, thus $i \notin L_\g$. Since the same can be said for $\d$ it follows by counterpositive that $L_\g \sse L_\e \cap L_\d$.  If $\vv V \vDash \g$, then  $\vv V = \vv W \join \vv Z$, contrary to the hypothesis that $\vv V$ is join irreducible. Hence $\vv V \not\vDash \g$ and so $\Delta$ has the finite intersection property.

So we can take an ultrafilter $U$ on $I$ containing $\Delta$; let $\alg B = \prod_{i\in I} \alg A/\th_i/U$. Since $\alg A/\th_i$ is subdirectly irreducible,  then by Lemma \ref{wellconnected} it is $\Gamma^{n}$-connected. But being $\Gamma^{n}$-connected is a first order property, so $\alg B$ is $\Gamma^{n}$-connected as well and clearly $\alg B \in \VV(\alg A)$. Moreover if an equation $\e$ fails in $\alg A$, then
$L_\e \in \Delta$, so $L_\e \in U$ and so $\alg B \not \vDash \e$.  This proves that $\VV(\alg B) = \VV(\alg A) = \vv V$ and so (3) holds.

Finally the proof that (3) implies (2) is similar to the one of Lemma \ref{KOlemma}. Suppose that $\vv V= \VV(\alg A)$ for some $\Gamma^n$-connected algebra $\alg A$, and that, by way of contradiction, $\vv V$ is not join irreducible, i.e. $\vv V = \vv W  \join \vv Z$ for some proper subvarieties $\vv W$ and $\vv Z$. Then there is an equation $p(\vuc xn)\geq 1$ holding in $\vv W$ but not in $\vv Z$ and an equation $q(\vuc ym) \geq 1$ holding in $\vv Z$ but not in $\vv W$. Clearly neither equation can hold in $\vv V$, so there are $\vuc an,\vuc bm \in A$ with $p(\vuc an) \not \ge 1$ and $q(\vuc bm) \not\ge 1$; since $\alg A$ is $\Gamma^{n}$-connected, there must be $\hat\g_{1}, \hat\g_{2}$, instances in $\alg A$ of iterated conjugates $\g_{1}, \g_{2} \in \Gamma^{n}$, such that $$\hat\g_{1}(p(\vuc an)) \join \hat\g_{2}(q(\vuc bm)) \not= 1.$$
hence $\vv V = \VV(\alg A) \not\vDash \g_{1}(p(\vuc xn)) \join \g_{2}(q(\vuc ym)) \approx 1$. On the other hand, since conjugates of elements in the positive cone of on algebra are also in the positive cone, $\g_{1}(p(\vuc xn)) \approx 1$ holds in $\vv W$ and $\g_{2}(q(\vuc ym)) \approx 1$ holds in $\vv Z$. Thus $\vv W,\vv Z \vDash \g_{1}(p(\vuc xn)) \join \g_{2}(q(\vuc ym)) \approx 1$ and since $\vv V = \vv W \join \vv Z$ we must have that
$\vv V \vDash\g_{1}(p(\vuc xn)) \join \g_{2}(q(\vuc ym)) \approx 1$. This is a contradiction, derived from the assumption that $\vv V$ was not join irreducible. Thus, also (2) and (3) are equivalent and the proof is completed.
\end{proof}

Observe that $\Gamma^0$-connectedness is weak well-connectedness and $\Gamma^0$-completeness is the so called  {\em weak Halld\'en completeness} for a substructural logic $\mathcal L$: if $\mathcal L \vdash (\phi \meet 1) \join (\psi \meet 1)$ then either $\mathcal L \vdash \phi$ or $\mathcal L \vdash \psi$. So our result extends both  Theorem 2.5 and 2.8 in \cite{KiharaOno2008}, giving an  (admittedly less meaningful) logical characterizations of the property of being join irreducible for a variety satisfying ($G_{n,n+1}$) for some $n \in \mathbb{N}$.

\section {Join irreducibility and subdirect irreducibility}

It is a straightforward consequence of Birkhoff's Theorem that if $\vv V$ is any variety of algebras that is strictly join irreducible, then $\vv V=\VV(\alg A)$ for some subdirectly irreducible algebra $\alg A$.
In \cite{KiharaOno2008} the authors asked if the converse might hold for varieties of residuated lattices.  We will see in  Section \ref{SJIbasichoops} below that not all the varieties generated by a subdirectly irreducible basic hoop are strictly join irreducible so in general the converse does not hold even in very friendly varieties.

Next we point out a corollary of Lemma \ref{wellconnected} and Theorem \ref{mainwell}.

\begin{corollary}\label{cor:mainwell} Let $\vv V$ be a variety of residuated lattices satisfying ($G_{n,n+1}$) for some $n \in \mathbb{N}$. If  there is a subdirectly irreducible algebra $\alg A$ with
$\vv V = \VV(\alg A)$, then $\vv V$ is join irreducible.
\end{corollary}

This is the analogous of Lemma 2.6(2) in \cite{KiharaOno2008} and again the authors asked if it was possible to invert it; we will show that our (more general) version is indeed invertible, thus
answering their question as well. First observe that:

\begin{lemma} \label{embed} Let $\vv V$ be a variety  of residuated lattices. If $\alg A \in \vv V$  is $\Gamma^n$-connected, then there exists a countable  $\Gamma^n$-connected  algebra $\alg B \in \vv V$  with $\VV(\alg A) = \VV(\alg B)$.
\end{lemma}
\begin{proof} Suppose that $\alg A$ is $\Gamma^n$-connected; it is clear that the class of $\Gamma^n$-connected  in $\VV(\alg A)$ is axiomatized by a set of first-order sentences, namely the
axioms defining $\VV(\alg A)$ and the universal sentence
$$
\forall x \forall y\, [\&B^n(x,y) \Rightarrow ((x \ge 1) \ \text{or}\ (y \ge 1))].
$$
Where  by $\&B^n(x,y)$ we mean the conjunction of the equations in
$B^{n}(x,y) = \{ \gamma_{1}(x) \join \gamma_{2}(y) \app  1 : \gamma_{1}, \gamma_{2} \in \Gamma^{n} \}$.
The downward version of the L\"owenheim-Skolem Theorem implies that if  $\alg A$  is  $\Gamma^n$-connected, then there is also a countable $\Gamma^n$-connected algebra $\alg B$ that is elementarily equivalent to $\alg A$.
In particular $\alg A$ and $\alg B$ must satisfy the same equations, so $\VV(\alg A) =\VV(\alg B)$.
\end{proof}

In \cite{AguzzoliBianchi2017} the authors proved an interesting result about $\mathsf{FL}_{ew}$-algebras: if a variety is generated by a well-connected algebra, then it is also generated by a subdirectly irreducible algebra, which implies that Lemma 2.6(2) in \cite{KiharaOno2008} can indeed be inverted.
By looking at their proofs it is clear that they do not  rely on zero-boundedness and commutativity. Generalizing it for non-integral structures takes a little more work and we need to clarify the context. If $\alg A$ is a residuated lattice
we can define a set of equations $D_\alg A$, called the {\bf diagram} of $\alg A$ in the following way: for every $a \in A$ we pick a variable $v_a$ and
$$
D_\alg A = \{v_a * v_b \leftrightarrow v_{a *b} \geq 1: a,b \in A, *\in\{\join,\meet,\imp,\lr,\rr\}\},
$$
where $a \leftrightarrow b = (a \backslash b) \land (b \backslash a)$.
It is clear that the diagram encodes the operation table of $\alg A$; moreover if $\alg A$ is countable then all the terms involved are elements of the free algebra over $\o$ generators of any variety to which $\alg A$ belongs.

So if $\alg A \in \vv V$, we can use the semantical consequence relation $\vDash$ on $\vv V$ and the expression $D_\alg A \vDash p\geq 1$ makes sense for any term $p$ in the language of residuated lattices. Likewise
$\alg B \in \vv V$ is a model of $D_\alg A\vDash p \geq 1$ if any homomorphism
$h : \alg F_\vv V (\omega) \longrightarrow \alg B$, say $h(v_a) = c_a$, such that
$c_a *c_b = c_{a*b}$ for all $a,b \in A$, is also such that $h(p) \geq 1$.  Now it is straightforward to check that
the class of models of $D_\alg A \vDash p \geq 1$ is closed under subalgebras and direct products; it follows that if  there is a countermodel in $\vv V$ of $D_\alg A \vDash p\geq 1$, then there is a subdirectly irreducible countermodel. This circle of ideas is the core of the following lemma.

\begin{lemma}\label{agbi} Let $\vv V$ be any variety of residuated lattices and let $\alg A$ be a countable $\Gamma^{n}$-connected algebra in $\vv V$; then there exists a subdirectly irreducible algebra $\alg B \in \vv V$ such that
$\alg A \in \SU(\alg B)$.
\end{lemma}
\begin{proof}
Let $\alg A$ be a countable $\Gamma^{n}$-connected algebra in $\vv V$ and let $V=\{v_a: a \in A\}$ be the set of variables defining the diagram $D_\alg A$.

Now we define a set of equations $\Psi$ not valid in $\alg A$ because of its $\Gamma^{n}$-connectedness. In particular, in $\Psi$ we write the meets of all possible joins of iterated conjugates of length $n$ of $m$ given variables, for each $m \in \mathbb{N}$. Notice that given $m \in \mathbb{N}$, such joins are in a finite number. As argued in \cite{Galatos2004} (p. 232), one can enumerate the indices of the iterated conjugates of a certain length and they will form an initial segment of the natural numbers. Let us call $\Gamma^{n}_{m}$ the finite set of all possible $m$-uples of iterated conjugates of length $n$, and $J_{m}$ the enumeration of the $m$-uples of indices of the conjugates in $\Gamma^{n}_{m}$. That is, each ${\bf j} \in J_{m}$ is of the kind ${\bf j} = (j_{1}, \ldots j_{m})$ and identifies an $m$-uple $(\gamma_{j_{1}}, \ldots, \gamma_{j_{m}}) \in \Gamma^{n}_{m}$. Then
\begin{equation*}
\begin{aligned}
\Psi = \{\bigwedge_{{\bf j} \in J_{m}} (\gamma_{j_{1}}(v_{a_{1}}) \lor \ldots \lor \gamma_{j_{m}}(v_{a_{m}})) \app 1&: \;m \in \mathbb{N}, (\g_{j_{1}}, \ldots, \g_{j_{m}}) \in \Gamma^{n}_{m},\\ &a_{i} \in A, a_{i} \not\geq 1, \mbox{ for } i = 1 \ldots m\}.
\end{aligned}
\end{equation*}
Notice that $\Psi$ is \emph{directed} in the following  sense:
$$ \mbox{if } \bigwedge_{j \in J} p_{j} \app 1 \in \Psi \mbox{ and}\bigwedge_{k \in K} q_{k} \app 1 \in \Psi, \;\;\mbox{ then }\;\; \bigwedge_{j \in J}\bigwedge_{k \in K} p_{j} \lor q_{k} \app 1 \in \Psi.$$
We call $\bigwedge_{j \in J}\bigwedge_{k \in K} p_{j} \lor q_{k} \app 1$ the \emph{upper bound} of  $\bigwedge_{j \in J} p_{j} \app 1$ and $ \bigwedge_{k \in K} q_{k} \app 1$, since both $$\bigwedge_{j \in J} p_{j} \leq \bigwedge_{j \in J}\bigwedge_{k \in K} p_{j} \lor q_{k} \mbox{ and } \bigwedge_{k \in K} q_{k} \leq \bigwedge_{j \in J}\bigwedge_{k \in K} p_{j} \lor q_{k} .$$ Moreover, $D_\alg A \not\models p \app 1$ for each $p \app 1 \in \Psi$, since $\alg A$ naturally gives a countermodel. Indeed,  with the assignment $g: \alg F_\vv V (\omega) \longrightarrow \alg A$ such that $g(v_{a}) = a$ for all $a \in A$, $\alg A$  satisfies all inequalities in $D_\alg A$ by construction, but it does not satisfy any of the equations $p \app 1 \in \Psi$: if $ p= \bigwedge_{j \in J} (\gamma_{j_{1}}(v_{a_{1}}) \lor \ldots \lor \gamma_{j_{m}}(v_{a_{m}})) \app 1$ is satisfied with respect to the assignment $g$, then all of the meetands $(\gamma_{j_{1}}(a_{1}) \lor \ldots \lor \gamma_{j_{m}}(a_{m})) = 1$, for all $(\g_{j_{1}}, \ldots, \g_{j_{m}}) \in \Gamma^{n}_{m}$,  and thus at least one of the elements $a_{i}\geq 1$, since $\alg A$ is $\Gamma^{n}$-connected. But by the definitions of $\Psi$ and $g$, $a_{i}\not\geq 1$ for all $i = 1 \ldots m$, a contradiction.

Let\footnote{This part of the proof is essentially the same as in Lemma 3.4 of \cite{CEGGMN}, which we rewrite in this more general setting.} $v$ be a variable different from all the $v_a$, and let
$$
D'= D_\alg A \cup\{p \rr v\ge 1: p \approx 1 \in \Psi\}.
$$
We claim that  $D' \not\models v \geq 1$. Suppose that $D' \models v \geq 1$, then by compactness there is a finite subset $D'' \subseteq D'$ such that $D'' \models v \geq 1$; let  $r \app 1$ be the upper bound of the finite set of equalities $\{p \approx 1 \in \Psi: p \rr v \geq 1 \in D'' \}$ (such upper bound exists as $\Psi$ is directed). Since $D_\alg A \not\models \psi$ for each $\psi \in \Psi$, in particular $D_\alg A \not\models r \app 1$;  so there is an algebra $\alg C \in \vv V$
and a homomorphism $k: \alg F_\vv V(\omega) \longrightarrow \alg C$ such that all identitities in $D_\alg A$ are satisfied in $\alg C$ but $k(r) \neq 1$, or equivalently $k(r) \not\ge 1$.  Now if we extend $k$ by setting $k(v) = k(r)$ then
$\alg C$ becomes a countermodel of  $D' \vDash v \geq 1$ as well, a contradiction. Thus $D' \not\models v \geq 1$. It follows that there is a subdirectly irreducible algebra $\alg B \in \vv V$ that is a countermodel of $D'\vDash v \geq 1$, with respect to an homomorphism $h : \alg F_\vv V (\omega) \longrightarrow \alg B$.

Observe that, since $D_\alg A \sse D'$,  $\alg B$ also satisfies all inequalities in $D_\alg A$ with respect to $h$, but  $h(p) \ne 1$ for all $p \approx 1 \in \Psi$. Indeed if it were $h(p) = 1$ it would follow that $h(v) \geq h(p) h(p \backslash v) \geq 1$, a contradiction.

Let's define $f: A \to B$ as $f(a) = h(v_{a})$, and show that it is an ambedding, which would settle the proof. Since  $\alg B$ satisfies $D_\alg A$ with respect to $h$, $f$ is a homomorphism.  To show that $f$ is injective suppose that $f(a) = f(b)$ given $a, b \in A$; then by residuation $1 \leq f(a \backslash b), 1 \leq f(b\backslash a)$.
Note that if $c \not\geq 1$, then $v_{c} \meet 1 \approx 1 \in \Psi$, thus $f(c)= h(v_{c}) \not\geq 1$; in fact if $h(v_c) \ge 1$ we would have $h(v_c \meet 1) = h(v_c) \meet 1 = 1$, a contradiction, as $\alg B$ satisfies no equality in $\Psi$ with respect to $h$. This implies that $1 \leq a \backslash b$ and $1 \leq b \backslash a$, thus $a = b$ and the proof is completed.

\end{proof}

\begin{theorem}\label{main3} Let $\vv V$ be a variety of residuated lattices that satisfies  ($G_{n,n+1}$) for some $n \in \mathbb{N}$; if $\vv V$ is join irreducible, then there is a subdirectly irreducible algebra $\alg B \in \vv V$ such that
$\VV(\alg B) = \vv V$.
\end{theorem}
\begin{proof}  Since $\vv V$ is join irreducible and satisfies  ($G_{n,n+1}$), by Theorem \ref{mainwell} there is a $\Gamma^{n}$-connected algebra $\alg A$ with $\VV(\alg A) = \vv V$. By Lemma \ref{embed} we may assume without loss of generality that $\alg A$ is countable; hence by Lemma \ref{agbi} there is a subdirectly irreducible $\alg B \in \vv V$ with $\alg A \in \SU(\alg B)$.
Clearly $\vv V = \VV(\alg A) =\VV(\alg B)$ and the thesis holds.
\end{proof}

Thus combining Theorem \ref{main3} and Corollary \ref{cor:mainwell} we obtain the following.
\begin{corollary}\label{main2} Let $\vv V$ be a variety of residuated lattices that satisfies  ($G_{n,n+1}$) for some $n \in \mathbb{N}$. Then $\vv V$ is join irreducible if and only if there is a subdirectly irreducible algebra $\alg A \in \vv V$ such that
$\VV(\alg A) = \vv V$.
\end{corollary}
Of course the same results we presented in this section hold for varieties of $\mathsf{FL}$-algebras. In particular then, the previous corollary applies to all varieties satisfying $(G)$, thus for example: normal varieties of $\mathsf{RL}$ and $\mathsf{FL}$, and thus $\mathsf{CRL}$, $\mathsf{CIRL}$, $\mathsf{FL}_{e}$, $\mathsf{FL}_{ew}$ (answering then to Kihara and Ono question); representable subvarieties of $\mathsf{RL}$ and $\mathsf{FL}$. Moreover, the theorem also applies to subvarieties of $\ell$-groups, since they satisfy ($G_{1,2}$).

\section{Representable varieties}\label{repr}

A residuated lattice is {\bf representable} if it is a subdirect product of totally ordered residuated lattices; a variety is representable if each of its members is representable, which is equivalent to say that each subdirectly irreducible member of the variety is totally ordered. Being representable for a variety is equationally definable; indeed for any $\vv V \in \mathsf{RL}$, $\vv V$ is representable if and only if $\vv V \vDash u\back((x \lor y)\back x)u \lor v((x \lor y)\back y)/v$ (see \cite{JipsenTsinakis2002}, \cite{BlountTsinakis2003}).

A  residuated lattice is {\bf prelinear} if it satisfies
$$
(x \lr y) \join (y \lr x) \ge 1.
$$
A variety of residuated lattices is prelinear if each of its members is prelinear.
Since any totally ordered residuated lattice is clearly prelinear, any representable variety is prelinear.

It is well-known that any representable variety satisfies (G) (see for instance Theorem 6.7 in \cite{BlountTsinakis2003}) so we may apply Theorem \ref{mainwell} to conclude that
that each join irreducible subvariety is generated by a single weakly well-connected algebra. Moreover, since clearly the underlying lattices of any representable residuated lattice is distributive
(as a subdirect product of chains) the concepts of well-connected and weakly well-connected coincide.
Moreover the following holds.

\begin{lemma}\label{prelinear} Let $\vv V$ be a prelinear variety of residuated lattices and let $\alg A \in \vv V$. Then the following are equivalent:
\begin{enumerate}
\item $\alg A$ is well-connected;
\item $\alg A$ is totally ordered.
\end{enumerate}
\end{lemma}
\begin{proof}
The implication (2) implies (1) clearly follows from the definition of well-connectedness. Suppose now that $\alg A$ is prelinear and well-connected. since it is prelinear for $a,b \in A$, $
(a \lr b) \join (b \lr a) \ge 1.$
Since it is well connected, either $a\lr b \ge 1$ or $b \lr a \ge 1$; so either $b\le a$ or $a \le b$ and $\alg A$ is totally ordered.
\end{proof}

If the variety we are dealing with is also integral, then we can use ordinal sums, of which we recall the definition for the reader's convenience.

Let $(I, \leq)$ be a totally ordered set. For all $i \in I$ let ${\bf A}_{i}$ be a totally ordered integral residuated ($\land$-semi)lattice, such that for $i \neq j$, $A_{i} \cap A_{j} = \{1\}$. Then $\displaystyle\bigoplus_{i \in I}{\bf A}_{i}$, the {\em ordinal sum} of the family $({\bf A}_{i})_{i \in I}$,  is the structure whose base set is $\bigcup_{i \in I}A_{i}$ and the operations are defined as follows:
$$
\begin{array}{lll}
x \to y &=&\left\{
\begin{array}{ll}
x \to^{A_{i}} y &\mbox{ if } x, y \in A_{i},\\
y & \mbox{ if } x \in A_{i} \mbox{ and } y \in A_{j} \mbox{ with } i > j\\
1 & \mbox{ if } x \in A_{i}\setminus\{1\} \mbox{ and } y \in A_{j} \mbox{ with } i < j.
\end{array}
\right.\\
&&\\
x \cdot y &=&\left\{
\begin{array}{ll}
x \cdot^{A_{i}} y &\mbox{ if } x, y \in A_{i},\\
y & \mbox{ if } x \in A_{i} \mbox{ and } y \in A_{j}\setminus\{1\} \mbox{ with } i > j\\
x & \mbox{ if } x \in A_{i}\setminus\{1\} \mbox{ and } y \in A_{j} \mbox{ with } i < j.
\end{array}
\right.\\
\end{array}\\
$$
It is easily seen that join and meet are defined by the ordering obtained by stacking the various algebras one over the other. Moreover we call {\bf sum-irreducible} an algebra that is not isomorphic to an ordinal sum of nontrivial components.
We recall that:
\begin{theorem} (\cite{Agliano2018b}, Theorem 3.2) Any integral residuated semilattice is the ordinal sum of sum-irreducible residuated semilattices.
\end{theorem}

In particular if $\vv V$ is integral and representable, then every strictly join irreducible subvariety of $\vv V$ is generated by a single totally ordered algebra $\alg A$, which is the ordinal sum of sum-irreducible
totally ordered algebras in $\vv V$. Therefore knowing the sum-irreducible algebras in a variety of integral and representable residuated lattices is paramount. Though this is usually a difficult task there are cases in which
it can be achieved.

A {\bf pseudohoop} is a an integral and divisible  residuated lattice; the variety of representable pseudohoops will be denoted by $\mathsf{RPsH}$.
A {\bf Wajsberg pseudohoop} is a representable pseudohoop satisfying
\begin{align*}
&(y\lr x)\rr y \app (x \lr y) \rr  x;\\
&y \lr(x \rr  y) \app x \lr (y \rr  x).
\end{align*}
It turns out that totally ordered  Wajsberg pseudohoops are precisely the totally ordered sum irreducible pseudohoops.  This implies that:

\begin{theorem} \label{Dvu} \cite{Dvurecenskij2007} Let $\alg A$ be a totally ordered  pseudohoop; then $\alg A = \bigoplus_{i\in I} \alg A_i$ where each $\alg A_i$ is a
(totally ordered) Wajsberg pseudohoop (and the decomposition is unique up to isomorphism of the components).
\end{theorem}

If $\alg A$ is a totally ordered pseudohoop and $\alg A= \bigoplus_{i \in I} \alg A_i$ is its decomposition into Wajsberg pseudohoops, the {\bf index} of $\alg A$ is $n$ if $|I|=n$ and
infinite if $I$ is infinite. Let's consider the equation
\begin{equation}
\bigwedge_{i=0}^{n-1} x_i\lr (x_i \lr (x_{i+1} \rr x_i))  \le \bigvee_{i=0}^n x_i \tag{$\lambda_n$}.
\end{equation}

\begin{lemma}\label{index}(Lemma 6.1 in \cite{Agliano2018b}) For any totally ordered pseudohoop $\alg A$, the index of $\alg A$ is less or equal to $n \in \mathbb{N}$ if and only if $\alg A \vDash \lambda_n$.
\end{lemma}

By Theorem \ref{mainwell} and Lemma \ref{prelinear} if $\alg A$ is a totally ordered pseudohoop, $\VV(\alg A)$ is join irreducible in $\Lambda(\mathsf{RPsH})$.

\begin{theorem}\label{thm:notsji}
Let $\alg A = \bigoplus_{i \in I} \alg A_i$ be a totally ordered pseudohoop; if the index of $\alg A$ is infinite, then $\VV(\alg A)$ is not strictly join irreducible.
\end{theorem}
\begin{proof} Let $C$ be the set of all  chains in $\VV(\alg A)$ having finite index; by Lemma \ref{index} $\VV(\alg C) \ne \VV(\alg A)$ for all $\alg C \in C$.
Suppose that an equation $p(\vuc xn) \app 1$ fails in $\alg A$; then there are $\vuc an \in A$ such that $p(\vuc an) < 1$. If $\alg C$ is the subalgebra of $\alg A$ generated by $\vuc an$,
then $\alg C \not\vDash p(\vuc xn) \app 1$ and $\alg C$ cannot have index greater then $n$. Hence $\alg C \in C$ and  $\bigvee_{\alg C \in C} \VV(\alg C) = \VV(\alg A)$; hence $\VV(\alg A)$ is not strictly join irreducible.
\end{proof}

So if $\alg A$ is a totally ordered pseudohoop and $\VV(\alg A)$ is strictly join irreducible in $\Lambda(\sf RPsH)$, then $\alg A$ must have finite index. Chains of finite index have been studied in \cite{AglianoMontagna2017} in a different context ($\mathsf{BL}$-algebras and basic hoops); we will show that the results therein extend to totally ordered pseudohoops in a straightforward way.

Let $\alg P$ be a poset; a {\bf join dense completion} of $\alg P$ is a complete lattice $\alg L$ with an order embedding $\a:\alg P \longrightarrow \alg L$ such that $\a(\alg P)$ is {\bf join dense} in $\alg L$, i.e. for every $x \in L$, $x = \bigvee\{\a(p): p \in P, \a(p) \le x\}$.
For any subvariety $\vv V$ of $\sf RPsH$ we define $$\mathcal F_\vv V = \{\alg C \in \vv V: \alg C \mbox{ is a chain of finite index}\}$$ and $$\Phi_\vv V  = \{\VV(\alg A): \alg A \in \mathcal F_\vv V\};$$ $\Phi_{\vv V}$ is clearly a poset under inclusion and moreover:

\begin{theorem}\label{joindense} For any subvariety $\vv V$ of $\sf RPsH$, $\Lambda(\vv V)$ is the join dense completion of $\Phi_\vv V$.
\end{theorem}
\begin{proof} Let $\alg A$ be a finitely generated totally ordered pseudohoop in $\vv V$; then,
 $\alg A$ has index $\le n$ and so $\alg A \in \mathcal F_\vv V$. Since every subvariety $\vv W$ of  $\vv V$  is generated by its totally ordered members and every algebra is in the variety generated by its finitely generated subalgebras,   every subvariety $\vv W$ of $\vv V$   is the supremum of all varieties $\VV(\alg A)$,
where $\alg A \in \mathcal F_\vv W$. We have thus shown that for every $\vv W \in  \Lambda(\vv V)$
$$
\vv W =\bigvee\{\VV(\alg A): \alg A \in \mathcal F_\vv W\}.
$$
This implies that $\Lambda(\vv V)$ is the join dense completion of $\Phi_\vv V$.
\end{proof}

\begin{corollary}\label{correp} Let $\vv V$ be any variety of representable pseudohoops; if $\Lambda(\vv V)$ is finite then
 $\vv V$  satisfies $\lambda_n$ for some $n \in \mathbb{N}$, and for all $\alg A \in \mathcal F_\vv V$, $\VV(\alg A)$ is strictly join irreducible in $\Lambda(\vv V)$.
\end{corollary}
\begin{proof} If $\Lambda (\vv V)$ is finite then it must satisfy $\lambda_n$ for some $n \in \mathbb{N}$; otherwise in $\vv V$ there would be chains of arbitrarily large index and they would generate infinitely many distinct varieties. Moreover if $\alg A \in \mathcal F_\vv V$, then by Theorem \ref{mainwell} and Lemma \ref{prelinear}, $\VV(\alg A)$ is join irreducible. Thus since $\Lambda(\vv V)$ is finite, $\VV(\alg A)$ is strictly join irreducible in $\Lambda(\vv V)$.
\end{proof}

If $\alg A, \alg B \in \mathcal F_{\sf RPsH}$ we define $\alg A \preceq \alg B$ iff $\HH\SU\PP_u(\alg A) \sse \HH\SU\PP_u(\alg B)$; then $\preceq$ is a preorder with associated equivalence relation $\equiv$ and by J\'onsson Lemma $\VV(\alg A) = \VV(\alg B)$ if and only $\alg A \equiv \alg B$.  It follows that to determine which varieties of representable pseudohoops generated by a chain of finite index are strictly join irreducible we may restrict to a representative of each $\equiv$-class. To do so we must have some ways of classifying the pseudo Wajsberg components in the decomposition.

\begin{theorem} \cite{Dvurecenskij2002}\cite{BahlsColeGalatos2003}Let $\alg A$ be any totally ordered Wajsberg  pseudohoop. If $\alg A$ is cancellative, then there is a totally ordered group $\alg G$ such that $\alg A$ is isomorphic with $\alg G^-$, the negative cone of $\alg G$. If $\alg A$ is bounded, then there is a totally ordered group with strong unit $u$, such that $\alg A \cong \Gamma(\alg G,u)$.
\end{theorem}

$\Gamma(\alg G,u)$ is often called the {\em Mundici's functor}, see \cite{Dvurecenskij2002} for a definition. Since by Lemma 4.5 in \cite{Agliano2018a} each Wajsberg pseudohoops is either bounded or cancellative, the theorem gives a complete representation of totally ordered Wajsberg pseudohoops.  However this is less useful than it seems, at least in  general. First totally ordered groups are common (every free group is totally orderable); second, if $\alg G$ is a totally ordered group we have no idea in general of what $\II\SU\PP_u(\alg G^-)$ is (except for one case, see next section). We believe that a careful analysis of the techniques employed in  \cite{DvurecenskijHolland2007} and \cite{DvurecenskijHolland2009} could lead to a better understanding of the problem; however we will not explore this path here, leaving it open for future investigations.
From now on we will deal only with commutative structures, first because it is possible to prove meaningful results about them and second because they are more suited for the applications we have in mind.

\section{Strictly join irreducible varieties of basic hoops}\label{SJIbasichoops}

From now on we will stay within the variety $\mathsf{CIRL}$ of commutative and integral residuated lattices; in this case $a \rr b = b \lr a$ and we will use the symbol $a \imp b$ for both.
For any commutative residuated lattice the concepts of prelinear and representable coincide (see for instance \cite{Agliano2018a}); a {\bf basic hoop} is a commutative and representable pseudohoop and a {\bf Wajsberg hoop} is a commutative Wajsberg pseudohoop.

By Theorem \ref{Dvu} any totally ordered basic hoop is the ordinal sum of Wajsberg hoops (and this is how the result was originally formulated in \cite{AglianoMontagna2003}); hence
by the same argument as before (Theorem \ref{thm:notsji}), a strictly join irreducible variety $\vv V$ of basic hoops must be $\VV(\alg A)$ for some chain of finite index $\alg A$. In the commutative case however we can choose representatives of the $\equiv$-classes in a very nice and uniform way.

We recall that given any totally ordered hoop of finite index $\alg A = \bigoplus_{i=1}^n \alg A_i$ we have \cite{AglianoMontagna2003}
$$
\HH\SU\PP_u(\alg A) = \bigcup_{i=1}^n( \bigoplus_{j=1}^{i-1} \II\SU\PP_u(\alg A_j) \oplus \HH\SU\PP_u(\alg A_i)).
$$
Therefore if $\alg A = \bigoplus_{i=1}^n \alg A_i$ and $\alg B = \bigoplus_{j=1}^m \alg B_i$ belong to the class of totally ordered basic hoops of finite index $\mathcal F_{\sf BH}$, then $\alg A \equiv \alg B$ if and only if $n=m$,   $\II\SU\PP_u(\alg A_i) = \II\SU\PP_u(\alg B_i)$ for $i=1,\dots,n-1$ and $\HH\SU\PP_u(\alg A_n) = \HH\SU\PP_u(\alg B_n)$.

First we need to straighten out some technical details; a {\bf $\mathsf{BL}$-algebra} is a zero-bounded basic hoop, i.e. a hoop with a nullary operation, denoted by $0$, satisfying  $0 \le x$; a {\bf Wajsberg algebra} is a zero-bounded Wajsberg hoop. It is clear that any zero-free subreduct of  a $\mathsf{BL}$-algebra is a basic hoop, and that any bounded (i.e. with a smallest element) basic hoop is a reduct of a $\mathsf{BL}$-algebra;  moreover the variety $\mathsf{BH}$ is the class of zero-free subreducts of $\mathsf{BL}$-algebras \cite{AFM}.
 The operator $\II\SU\PP_u$ on Wajsberg hoops has been extensively studied first in \cite{Gispert2002} and then in \cite{AglianoMontagna2003};
really the results in \cite{Gispert2002} are about  MV-algebras but we can use them anyway and this is why.
MV-algebras are termwise equivalent to Wajsberg algebras which in turn are polynomially equivalent to bounded Wajsberg hoops. The following lemma is easy to show, and is often implicitly used in the study of Wajsberg hoops and Wajsberg algebras.
\begin{lemma}
Let $\OO$ be a class operator that is a composition of $\II,\HH,\SU,\PP,$ $\PP_u$. Let
$\alg A, \alg B$ be Wajsberg chains and $\alg A_0,\alg B_0$ be their Wajsberg hoop reducts. Then $\OO(\alg A) \sse \OO(\alg B)$ if and only if $\OO(\alg A_0) \subseteq \OO(\alg B_0)$.
\end{lemma}
This allows us to consider totally ordered bounded Wajsberg hoops {\em as if they were} Wajsberg algebras. In other words to check that $\II\SU\PP_u(\alg A) = \II\SU\PP_u(\alg B)$ when they are both bounded totally ordered Wajsberg hoops is enough to check the equality by considering them as Wajsberg algebras.
Since a totally order Wajsberg hoop is either bounded or cancellative \cite{BlokFerr2000} we can use Gispert's results  in \cite{Gispert2002} for Wajsberg algebras and integrate them with the cancellative case.

In \cite{AglianoMontagna2017} an algorithm is given to choose a representative for each $\equiv$-class of finite ordinal sums of $\mathsf{BL}$-algebras; while we reassure that the same argument goes through for basic hoops, in this case it is more useful to present a slightly finer classification, based on some ideas in \cite{AguzzoliBianchi2020}.
Let $\alg G$ be a lattice ordered abelian group; by \cite{Mun1986}, if $u$ is a strong unit of $\alg G$ we can construct a bounded Wajsberg hoop $\Gamma(\alg G,u) = \la [0,u], \imp,\cdot, 0, u\ra$ where $ab = \op{max}\{a+b-u,0\}$ and $a \imp b = \op{max}\{u-a+b,u\}$. Let now  $\mathbb Z \times_l \mathbb Z$ denote the lexicographic product of two copies of $\mathbb Z$. In other
words, the universe is the cartesian product, the group operations are defined componentwise
and the ordering is the lexicographic ordering (w.r.t. the natural ordering of $\mathbb Z$); then  $\mathbb Z \times_l \mathbb Z$ is a totally ordered abelian group and we can apply  $\Gamma$ to it. We now define some useful Wajsberg chains:
\begin{enumerate}
\itemb the finite Wajsberg chain with $n+1$ elements $\alg{\L}_n = \Gamma(\mathbb Z,n)$;
\itemb the infinite Wajsberg chain $\alg{\L}_n^\infty = \Gamma(\mathbb Z\times_l \mathbb Z, (n,0))$;
\itemb the infinite Wajsberg chain $\alg {\L}_{n,k}=\Gamma(\mathbb Z\times_l \mathbb Z, (n,k))$;
\itemb the infinite bounded Wajsberg chain $[0,1]_{\alg{\L}} = \Gamma(\mathbf R,1)$, i.e.  the real interval with operations induced by the {\em Wajsberg $t$-norm}. i.e
$xy = \op{max}(x+y-1,0)$, $x \imp y = \op{min}(1+x-y,1)$;
\itemb the infinite bounded Wajsberg chain $\alg Q = \Gamma(\mathbb Q,1) = \mathbb Q \cap [0,1]_{\alg L}$;
\itemb if $U$ is a set of positive integers we denote by $\alg Q(U)$ the subalgebra of $\alg Q$ generated by $\{\alg {\L}_k: k \in U\}$, where $\alg {\L}_k$ is seen as an algebra with universe $\{\frac{0}{n}, \frac{1}{n}, \ldots, \frac{n}{n}\}$;
\itemb the unbounded Wajsberg chain $\alg C_\omega$ that has as universe the free group on one generator, where the product is the group product and $a^l \imp a^m = a^{\op{max}(m-l,0)}$;
\itemb  finally we fix once and for all an irrational number $\a \in [0,1]$ and we  let $X$ be the totally ordered dense subgroup of $\mathbb R$ generated by $\a$ and $1$; then $\alg S_n= \Gamma(X,n)$.
\end{enumerate}
The {\bf radical} of a bounded Wajsberg hoop $\alg A$, in symbols $\op{Rad}(\alg A)$, is the intersection of the maximal filters of $\alg A$; it is easy to see that $\op{Rad}(\alg A)$ is a cancellative basic subhoop of $\alg A$. We say that a bounded Wajsberg hoop $\alg A$ {\bf has rank $n$},
 if $\alg A/\op{Rad}(\alg A) \cong \alg{\L}_n$, otherwise it has infinite rank. For any bounded Wajsberg hoop $\alg A$, $d_\alg A$, called the {\bf divisibility index}, is the maximum $k$ such that $\alg{\L}_k$ is embeddable in $\alg A$ if any, otherwise $d_\alg A = \infty$.

\begin{lemma}\label{isp}
\begin{enumerate}
\item $\alg {\L}_n$ has rank $n$ and divisibility index $n$.
\item For any $k\ge 0$, $\alg{\L}_{n,k}$ has rank $n$ and $d_{\alg {\L}_{n,k}} = \gcd(n,k)$; in particular $d_{\alg L^\infty_n} = n$.
\item $\alg S_n$ has infinite rank and $\alg {\L}_k \in \SU(\alg S_n)$ if and only if $k \mathrel{|} n$; hence $d_{\alg S_n} =n$.
\item If $\alg A$ is a nontrivial totally ordered cancellative hoop  then $\II \SU\PP_u(\alg A) = \II\SU\PP_u(\alg C_\o)$.
\item  If $\alg A$ is a non-simple bounded Wajsberg chain of finite rank $n$, then $\II\SU\PP_u(\alg A) = \II\SU\PP_u(\alg {\L}^\infty_n)$ if and only if $d_\alg A= n$.
\item  If $\alg A$ is a non-simple bounded Wajsberg chain of finite rank
$k$,  then $d_\alg A$ divides $k$, and $\II\SU \PP_u(\alg A)= \II\SU\PP_u(\alg {\L}_{k,d_\alg A})$.
\item If $\alg A$ is a Wajsberg chain of infinite rank then
 $\II\SU\PP_u (\alg A) = \II\SU\PP_u(\alg S_n)$ if and only if $d_\alg A = n$.
\item If $\alg A$ is a bounded Wajsberg chain of infinite rank and $d_\alg A = \infty$, then
$\II\SU\PP_u(\alg A) = \II\SU\PP_u(\alg Q(U))$ where $U= \{k : \alg{\L}_k \in \SU(\alg A)\}$.
\item If $\alg A$ is an infinite subdirectly irreducible Wajsberg algebra
of finite rank $k$, then $\HH \SU \PP_u(\alg A)=\HH \SU
\PP_u(\alg {\L}_{k,k})$.
\item If $\alg A$ is a subdirectly irreducible Wajsberg algebra of
infinite rank, then $\HH \SU\PP_u(\alg A) = \HH \SU \PP_u([0,1]_{\alg {\L}})$.
\end{enumerate}
\end{lemma}
\begin{proof} Claims (1)  and (3) are  obvious, while  (2), (4) and (5) are in \cite{AglianoMontagna2003}.
The proofs of (6), (7) and (8) are in \cite{Gispert2002}.

For (9) note that both $\alg A$ and $\alg {\L}_{k,k}$ are subdirectly
irreducible and generate the same variety, namely the variety of
Wajsberg algebras  $\alg A$ such that $d_\alg A$ divides $k$ (see Lemmas 6.1 and 6.3 in \cite{AglianoMontagna2003}).
For (10)  it is well-known that any two algebras of infinite rank
generate the entire variety of Wajsberg algebras (see for instance Theorem 2.5 in \cite{AglianoPanti1999}).
Since any variety of
Wajsberg hoops is congruence distributive, (9) and (10) follow from J\'onssons's Lemma.
\end{proof}

\begin{theorem}\label{classfinitechains} Let $\alg A = \bigoplus_{i=1}^n \alg A_{i}$ be any totally ordered basic hoop of finite index; then there is an unique (up to $\equiv$) $\alg B \in \mathcal F_{\sf BH}$ such that $\alg A \equiv \alg B$. More precisely
if $\alg B = \bigoplus_{i=1}^n \alg B_i$, then
$$
\alg B_i = \left\{
            \begin{array}{ll}
              \alg C_\omega, & \hbox{if $\alg A_i$ is cancellative;} \\
              \alg{\L}_n, & \hbox{if $\alg A_i \cong \alg{\L}_n$;} \\
              \alg{\L}_{k,d}, & \hbox{if $\alg A_i$ is infinite, has rank $k$ and $d=d_{\alg A_i}$;} \\
              \alg Q(U), & \hbox{if $\alg A_i$ has infinite rank and }\\
&\hbox{ $U=\{k: \alg{\L}_k\in \SU(\alg A)\}$ is infinite;}\\
              \alg S_k, & \hbox{if  $\alg A_i$ has infinite rank and $d_{\alg A_i} = k$}.
            \end{array}
          \right.
$$
for $i <n$ and
$$
\alg B_n = \left\{
             \begin{array}{ll}
               \alg{\L}_n, & \hbox{if $\alg A_n \cong \alg{\L}_n$;} \\
               \alg {\L}_{k,k}, & \hbox{if $\alg A_n$ is infinite and has rank $k$;} \\
               \alg C_\omega, & \hbox{if $\alg A_n$ is cancellative;} \\
               {[0,1]_{\alg {\L}}}, & \hbox{if $\alg A_n$ has infinite rank.}
             \end{array}
           \right.
$$
\end{theorem}

\begin{proof} That $\HH\SU\PP_u(\alg A) = \HH\SU\PP_u(\alg B)$ follows from Lemma \ref{isp} and the description of $\HH\SU\PP_u(\alg A)$.
To prove uniqueness  let $\alg C = \bigoplus_{i=1}^m\alg C_i \in \mathcal F_{\sf BH}$; if $m\ne n$, then  $\II\SU\PP_u(\alg B) \ne \II\SU\PP_u(\alg C)$. If they have the same index $n$ and they are different, then they must differ on some Wajsberg component. The first possibility is that there exists an $i < n$ such that
\begin{enumerate}
\itemb $\alg B_i$ is cancellative and $\alg C_i$ is not (or viceversa);
\itemb for some $k$, $\alg B_i$ has at most $k$ elements and $\alg C_i$ has more (or viceversa);
\itemb $\alg {\L}_k$ is embeddable in $\alg B_i$ and it is not embeddable in $\alg C_i$ (or viceversa);
\itemb  $\{k: \alg {\L}_k\ \text{is embeddable in}\ \alg B_i\}$ is infinite  and $\{k: \alg {\L}_k$ is embeddable in $\alg C_i\}$ is not (or viceversa);
\itemb $\alg B_i$ and $\alg C_i$ have different divisibility index.
\end{enumerate}
But it is  easy to see that any of this conditions implies $\II\SU\PP_u(\alg B_i) \ne \II\SU\PP_u(\alg C_i)$, so $\HH\SU\PP_u(\alg B) \ne \HH\SU\PP_u(\alg C)$.
If this is not the case, then $\alg B_n$ and $\alg C_n$ must differ; a similar argument shows that in this case $\HH\SU\PP_u(\alg B_n) \ne \HH\SU\PP_u(\alg C_n)$.
Hence $\alg B$ is unique up to $\equiv$.
\end{proof}

Now we can characterize the subvarieties of $\mathsf{BH}$ that are strictly join irreducible (the reader can compare this with the analogous result for $\mathsf{BL}$-algebras proved in \cite{AguzzoliBianchi2020}).

\begin{theorem}\label{thm:sjibasic} A subvariety $\vv V$ of $\sf BH$ is strictly join irreducible if and only if it is equal to $\VV(\alg A)$ for some $\alg A \in \mathcal F_\vv V$ such that its Wajsberg components are either cancellative
or are bounded and have finite rank.
\end{theorem}
\begin{proof} Let's prove necessity first. We can work modulo $\equiv$ i.e. we may suppose that $\alg A = \bigoplus_{i=1}^n \alg A_i$ has the ``canonical form'' suggested by Theorem \ref{classfinitechains}. So suppose that $\alg A_n=[0,1]_{\alg {\L}}$ and
for any $k\in \mathbb N$ let $\alg C_k = \bigoplus_{i=1}^{n-1} \alg A_i \oplus \alg{\L}_k$; then clearly $\alg C_k \not \equiv \alg A$ for all $k$. However
$\HH\SU\PP_u(\{\alg C_k: k \in \mathbb N\} = \HH\SU\PP_u(\alg A)$ since it is well known that $\HH\SU\PP_u(\{\alg{\L}_k: k \in \mathbb N\}) = \HH\SU\PP_u([0,1])$.
Hence $\bigvee_{k \in \mathbb N} \VV(\alg C_k) = \VV(\alg A)$ and $\VV(\alg A)$ is not strictly join irreducible.

Next suppose that for some $i<n$, $\alg A_i$ is bounded and has infinite rank; modulo $\equiv$ either $\alg A_i =  \alg Q(U)$, where $U=\{k: \alg{\L}_k\in \SU(\alg A)\}$ is infinite, or else $\alg A = \alg S_m$ for some $m$.
In the first case for $k \in U$ we let $\alg C_k$ to be the chain in which we have replaced the i-th component of $\alg A$ with $\alg {\L}_k$; clearly $\alg C_k \not\equiv \alg A$ for all $k$. However, since $\alg Q(U)$ is generated by
$\{\alg L_k: k \in U\}$ any equation that fails in $\alg A$ must fail in $\alg C_k$ for some $k$; this proves that $\bigvee_{k \in U}\VV(\alg C_k) = \VV(\alg A)$ and so $\VV(\alg A)$ is not strictly join irreducible.

For the second case we let $\alg C_k$ to be the chain in which we have replaced the i-th component of $\alg A$ with $\alg {\L}_{m,k}$; again clearly $\alg C_k \not\equiv \alg A$ for all $k$ and a very similar argument to the one above shows that $\bigvee_k \VV(\alg C_k) =\VV(\alg A)$ (but see Proposition 8 in \cite{AguzzoliBianchi2020} for details). Hence also in this case $\VV(\alg A)$ is not strictly join irreducible and the proof of necessity is
finished.

 Let $\vv V = \VV(\alg A)$ where $\alg A$ satisfies the hypotheses of the theorem; since $\alg A$ is well-connected by Lemma \ref{prelinear}, then $\vv V$ is join irreducible in $\Lambda (\sf BH)$ by Theorem \ref{mainwell}.
Therefore if we show that $\Lambda(\vv V)$ is finite, then the thesis will hold.  Now by Theorem \ref{joindense} it is enough to show that $\Phi_\vv V$ is finite, i.e. that $\mathcal F_\vv V/\mathop{\equiv}$ is a finite set.
Now any ``canonical" representative in $\mathcal F_\vv V/\mathop{\equiv}$ must be a chain of index $\le$ than the index of $\alg A$ and whose components are either $\alg C_\o$ or whose rank cannot exceed the maximum rank of the non cancellative components of $\alg A$. Clearly there are only finitely many choices, so $\mathcal F_\vv V/\mathop{\equiv}$ is finite as wished.
\end{proof}

\begin{corollary}\label{corbh} Let $\vv V$ be any variety of basic hoops; then the following are equivalent:
\begin{enumerate}
\item $\Lambda(\vv V)$ is finite;
\item  $\vv V$  satisfies $\lambda_n$ and for all $\alg A \in \mathcal F_\vv V$, $\VV(\alg A)$ is strictly join irreducible in $\Lambda(\vv V)$.
\end{enumerate}
\end{corollary}

\section{Some examples}

The only subvariety of $\sf BH$ satisfying $\lambda_1$ is of course the variety $\sf WH$ of Wajsberg hoops; $\Lambda(\sf WH)$ has been totally described in \cite{AglianoPanti1999}. It turns out that a proper subvariety of $\sf WH$ is generated by finitely many totally ordered hoops among $\alg C_\omega$ and $\{\alg{\L}_n, \alg{\L}_n^\infty: n \in \mathbb N\}$. This (by J\'onnson Lemma) implies that for any proper subvariety $\vv V$ of $\sf WH$,
$\Lambda(\vv V)$ is finite and hence by Corollary \ref{corbh}, $\vv V$ is strictly join irreducible if and only if it is join irreducible if and only if it is generated by a single totally ordered Wajsberg hoop that is either cancellative or has finite rank (and so it is either $\alg{\L}_n$ or $\alg{\L}_n^\infty$ for some $n$).

Let $\mathsf{C} = \VV(\alg C_\o)$ be the variety of cancellative hoops and let for $n \in \mathbb N$, $\mathsf{W}_n = \VV(\alg{\L}_n)$ and $\mathsf{W}^\infty_n = \VV(\alg{\L}^\infty_n)$.  In Figure \ref{wajsberg} we see
a picture of $\Lambda(\vv W^\infty_6)$ where we have labeled the subdirectly irreducible elements.

\begin{figure}[htbp]
\begin{center}
\begin{tikzpicture}[scale=.9]
\draw (-1,1) --(0,2) -- (1,1) -- (0,0) -- (-1,1) -- (-2,2) -- (-1,3) -- (0,2)-- (1,3) -- (0,4) -- (-1,3);
\draw (-1,2) -- (2,5) -- (1,6) -- (-2,3) -- (-1,2);
\draw (0,3) -- (-1,4);
\draw (1,4) -- (0,5);
\draw (1,6) -- (0,7) -- (-1, 6) -- (0,5);
\draw (-1,1) -- (-1,2);
\draw (0,2) -- (0,3);
\draw (1,3) -- (1,4);
\draw (-2,2) -- (-2,3.5) -- (1,6.5) -- (1,6);
\draw (-1,3) -- (-1,4.5);
\draw (0,4) -- (0,5.5);
\draw (0,5.5) -- (-1, 6.5) -- (0,7.5) -- (1,6.5);
\draw (0,4) -- (-1,5) -- (-1,6.5);
\draw (1,6) -- (1,6.5);
\draw (0,7) -- (0,8.5);
\draw[fill] (0,0) circle [radius=0.05];
\draw[fill] (-1,1) circle [radius=0.05];
\draw[fill] (1,1) circle [radius=0.05];
\draw[fill] (-2,2) circle [radius=0.05];
\draw[fill] (-1,2) circle [radius=0.05];
\draw[fill] (0,2) circle [radius=0.05];
\draw[fill] (-2,3) circle [radius=0.05];
\draw[fill] (-1,3) circle [radius=0.05];
\draw[fill] (0,3) circle [radius=0.05];
\draw[fill] (1,3) circle [radius=0.05];
\draw[fill] (-2,3.5) circle [radius=0.05];
\draw[fill] (-1,4) circle [radius=0.05];
\draw[fill] (0,4) circle [radius=0.05];
\draw[fill] (1,4) circle [radius=0.05];
\draw[fill] (-1,4.5) circle [radius=0.05];
\draw[fill] (-1,5) circle [radius=0.05];
\draw[fill] (0,5) circle [radius=0.05];
\draw[fill] (2,5) circle [radius=0.05];
\draw[fill] (0,5.5) circle [radius=0.05];
\draw[fill] (-1,6) circle [radius=0.05];
\draw[fill] (1,6) circle [radius=0.05];
\draw[fill] (-1,6.5) circle [radius=0.05];
\draw[fill] (1,6.5) circle [radius=0.05];
\draw[fill] (0,7) circle [radius=0.05];
\draw[fill] (0,7.5) circle [radius=0.05];
\draw[fill] (0,8.5) circle [radius=0.05];
\node[left] at (-1,1) {\footnotesize  $\vv W_1$};
\node[right] at (1,1) {\footnotesize  $\vv C$};
\node[left] at (-2,2) {\footnotesize  $\vv W_2$};
\node[right] at (-1,2) {\footnotesize  $\vv W_3$};
\node[right] at (1,3) {\footnotesize  $\vv W_1^\infty$};
\node[left] at (-2,3.5) {\footnotesize  $\vv W_6$};
\node[left] at (-1,5) {\footnotesize  $\vv W_2^\infty$};
\node[right] at (2,5) {\footnotesize  $\vv W_3^\infty$};
\node[right] at (0,8.5) {\footnotesize  $\vv W_6^\infty$};
\end{tikzpicture}
\end{center}
\caption{$\Lambda(\vv W^\infty_6)$}\label{wajsberg}
\end{figure}
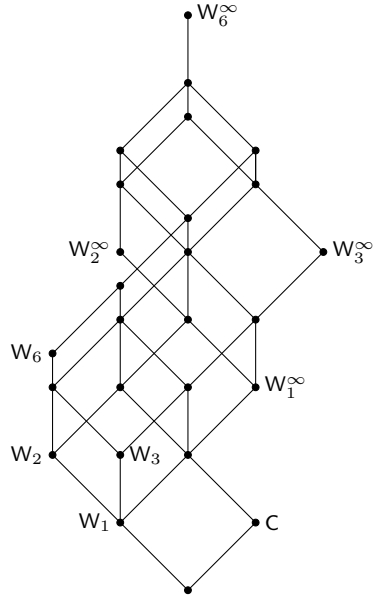

Now we would like to examine a larger lattice of subvarieties and we need some definitions.
A {\bf G\"odel hoop} is an idempotent basic hoop. G\"odel hoops are termwise equivalent to {\em relative Stone lattices} \cite{HechtKatrinak1972} and the variety is denote by $\sf GH$;  it is well-known that the only subdirectly irreducible in $\sf GH$ is $\mathbf 2$ (the two element residuated lattice) and that $\Lambda(\mathsf{GH})$ is a chain of type $\o+1$. More precisely, if $\alg G_n = \bigoplus_{i=1}^n \mathbf 2$, then the proper subvarieties of $\sf GH$ are $\mathsf{G}_n = \VV(\alg G_n)$.

 A {\bf product hoop} is a basic hoop satisfying
$$
(y \imp z) \join ((y \imp xy) \imp x) \app 1;
$$
The variety $\sf PH$ of product hoops has been examined in \cite{AFM}; it turns out that $\Lambda(\sf PH)$ has five elements and the proper non trivial subvarieties are $\sf C$, $\sf G_1$ and $\mathsf{C} \join \mathsf {G}_1$.
Moreover $\mathsf{PH} = \VV(\mathbf 2 \oplus \alg C_\o)$ so it is strictly join irreducible in $\Lambda(\mathsf{BH})$.
Let $\vv L = \mathsf{WH} \join \mathsf{PH} \join \mathsf{GH}$; we want to describe $\Lambda(\vv L)$.

Since basic hoops are congruence distributive, $\Lambda(\mathsf{BH})$ is distributive and since it is also dually algebraic it satisfies the infinite
distributive law
$$
x \join \bigwedge_{i\in I} y_i \approx \bigwedge_{i\in I}(x \join y_i).
$$
Let $\alg L$ be any lattice; an {\bf interval} is the set $[a,b] = \{c : a \le c \le b\}$. An interval $[a,b]$ is an {\em upper transpose}
of and interval $[c,d]$ (and $[c,d]$ is a {\em lower transpose} of $[a,b]$)
if $a \join d = b$ and $a \meet d = c$. Two intervals $[a,b]$, $[c,d]$ are
{\em $n$-step projective} if there exists intervals
$[a,b]=[x_0,y_0],[x_1,y_1],\dots,[x_n,y_n]=[c,d]$ such that $[x_i,y_i]$ is an
upper or lower transpose of $[x_{i+1},y_{i+1}]$. Two intervals are {\em
projective} if they are $n$-step projective for some $n$. It is a nontrivial
fact that in a distributive lattice ``projectivity = $2$-step projectivity".
 Moreover we say that $b$ is a {\em cover}
of $a$ and we write $a \prec b$ if $[a,b]=\{a,b\}$.

Let us apply these facts to $\vv L$. First observe that
$$
\mathsf{WH} \cap \mathsf{PH} = \mathsf{C} \join \mathsf{G}_1 = \VV(\mathbf 2,\alg C_\o);
$$
to simplify the notation we will call that variety $\mathsf{CPH}$.
Since $\mathsf{WH} \cap \mathsf{PH} = \mathsf{CPH}$ and $\mathsf{CPH}\prec \mathsf{PH}$
we have that $\mathsf{WH} \prec \mathsf{WH} \join \mathsf{PH}$ and
$[\mathsf{CPH},\mathsf{WH}] \cong [\mathsf{PH},\mathsf{WH} \join \mathsf{PH}]$.
The same conclusion holds if we substitute $\mathsf{WH}$ with  any variety of Wajsberg
hoops $\vv W$.

\begin{lemma} Let $\vv W$ any variety of Wajsberg hoops.
If  $\vv V \in [\mathsf{PH},\vv W \join\mathsf{PH}]$ then there exists a variety
$\vv W' \sse \vv W$ such that $\vv W' \join \mathsf{PH} = \vv V$.
If $\vv V \in [\mathsf{CPH},\vv W \join \mathsf{PH}]$, then either
$\mathsf{PH} \le \vv V$ or $\vv V \le \mathsf{WH}$.
\end{lemma}

\begin{proof} Let $\vv W'=\mathsf{WH} \cap \vv V$; then
\begin{align*}
\vv W' \join \mathsf{PH} &= (\mathsf{WH} \cap \vv V) \join \mathsf{PH} \\
&= (\mathsf{WH} \lor \mathsf{PH}) \cap (\vv V \lor \mathsf{PH}) =  (\mathsf{WH} \lor \mathsf{PH}) \cap \vv V = \vv V.
\end{align*}
For the second, let $\vv W' = \mathsf{WH} \cap \vv V$. By distributivity we see at
once that $[\vv W',\vv W' \join\mathsf{PH}]$ and $[\mathsf{CPH},\mathsf{PH}]$ are isomorphic.
It follows that $\vv W' \prec \vv W'\join \mathsf{PH}$ and hence either
$\vv V = \vv W'$ or $\vv V = \vv W' \join \mathsf{PH}$. The former
implies $\vv V \le \mathsf{WH}$ and from the latter
\begin{align*}
\vv V = \vv W' \join\mathsf{PH} &= (\mathsf{WH} \cap \vv V) \join \mathsf{PH}\\
&=(\mathsf{WH}\cap\mathsf{PH})\join(\vv V \join \mathsf{PH}) = \mathsf{CPH} \join(\vv V \join \mathsf{PH})
=\vv V \join \mathsf{PH}.
\end{align*}
Thus $\mathsf{PH} \le \vv V$.
\end{proof}

The same argument can be applied as it is to the intervals
 $[\vv W_2,\mathsf{GH}]$, $[\mathsf{CPH},\mathsf{CPH} \join \mathsf{GH}]$ and $[\mathsf{PH},\mathsf{PH} \join \mathsf{GH}]$.
Then to $[\mathsf{CPH},\mathsf{GH}\join\mathsf{PH}]$ and $[\mathsf{WH},\vv L]$,
and finally to  $[\mathsf{CPH},\mathsf{WH}\join\mathsf{PH}]$ and
$[\mathsf{GH},\vv L]$. Repeating this process for
any variety of Wajsberg hoops or G\"odel hoops we obtain a picture of the lattice of subvarieties of $\vv L$ (Figure \ref{l}).

\begin{figure}[htbp]
\begin{center}
\begin{tikzpicture}[scale=.9]
\draw (0,3) -- (0,2) -- (1,1) -- (0,0) -- (0,2);
\draw (-5,5) -- (-5,6) -- (-3,8) -- (-3,7);
\draw (2,3) -- (2,5);
\draw (5,6) -- (5,8);
\draw (0,10) -- (0,11);

\draw [thick][dotted] (0,1) to [out=165,in=280] (-5,5);
\draw [thick][dotted](0,2) to [out=165,in=285] (-5,5);
\draw [thick][dotted](0,2) to [out=130,in=340] (-5,5);
\draw [thick][dotted](0,3) to [out=165,in=285] (-5,6);
\draw [thick][dotted](0,3) to [out=130,in=340] (-5,6);

\draw [thick][dotted](2,3) to [out=165,in=280] (-3,7);
\draw [thick][dotted](2,4) to [out=165,in=285] (-3,7);
\draw [thick][dotted](2,4) to [out=130,in=340] (-3,7);
\draw [thick][dotted](2,5) to [out=165,in=285] (-3,8);
\draw [thick][dotted](2,5) to [out=130,in=340] (-3,8);

\draw [thick][dotted](5,6) to [out=165,in=280] (0,10);
\draw [thick][dotted](5,7) to [out=165,in=285] (0,10);
\draw [thick][dotted](5,7) to [out=130,in=340] (0,10);
\draw [thick][dotted](5,8) to [out=165,in=285] (0,11);
\draw [thick][dotted](5,8) to [out=130,in=340] (0,11);

\draw (0,1) -- (2,3);
\draw[dashed] (2,3) --(5,6);
\draw (0,2) -- (2,4);
\draw[dashed] (2,4) --(5,7);
\draw (0,3) -- (2,5);
\draw[dashed] (2,5) --(5,8);
\draw (-5,5) -- (-3,7);
\draw[dashed] (-3,7) --(0,10);
\draw[dashed] (-3,8) --(0,11);

\draw[fill] (0,0) circle [radius=0.05];
\draw[fill] (0,1) circle [radius=0.05];
\draw[fill] (0,2) circle [radius=0.05];
\draw[fill] (0,3) circle [radius=0.05];
\draw[fill] (1,1) circle [radius=0.05];
\draw[fill] (-5,5) circle [radius=0.05];
\draw[fill] (-5,6) circle [radius=0.05];

\draw[fill] (2,3) circle [radius=0.05];
\draw[fill] (2,4) circle [radius=0.05];
\draw[fill] (2,5) circle [radius=0.05];
\draw[fill] (-3,7) circle [radius=0.05];
\draw[fill] (-3,8) circle [radius=0.05];

\draw[fill] (5,6) circle [radius=0.05];
\draw[fill] (5,7) circle [radius=0.05];
\draw[fill] (5,8) circle [radius=0.05];
\draw[fill] (0,10) circle [radius=0.05];
\draw[fill] (0,11) circle [radius=0.05];

\node[left] at (0,.8) {\footnotesize  $\vv G_1$};
\node[right] at (1,.9) {\footnotesize  $\vv C$};
\node[right] at (0,3) {\footnotesize  $\mathsf{PH}$};
\node[right] at (2,2.9) {\footnotesize  $\vv G_2$};
\node[right] at (5,6) {\footnotesize  $\mathsf{GH}$};
\node[right] at (5,8) {\footnotesize  $\mathsf{GH} \join \mathsf{PH}$};
\node[left] at (-5,5) {\footnotesize  $\mathsf{WH}$};
\node[left] at (-5,6) {\footnotesize  $\mathsf{WH} \join \mathsf{PH}$};
\node[left] at (0,10) {\footnotesize  $\mathsf{WH} \join \mathsf{GH}$};
\node[above] at (0,11) {\footnotesize  $\mathsf{L}$};
\end{tikzpicture}
\end{center}
\caption{The lattice of subvarieties of $\vv L$\label{l}}
\end{figure}
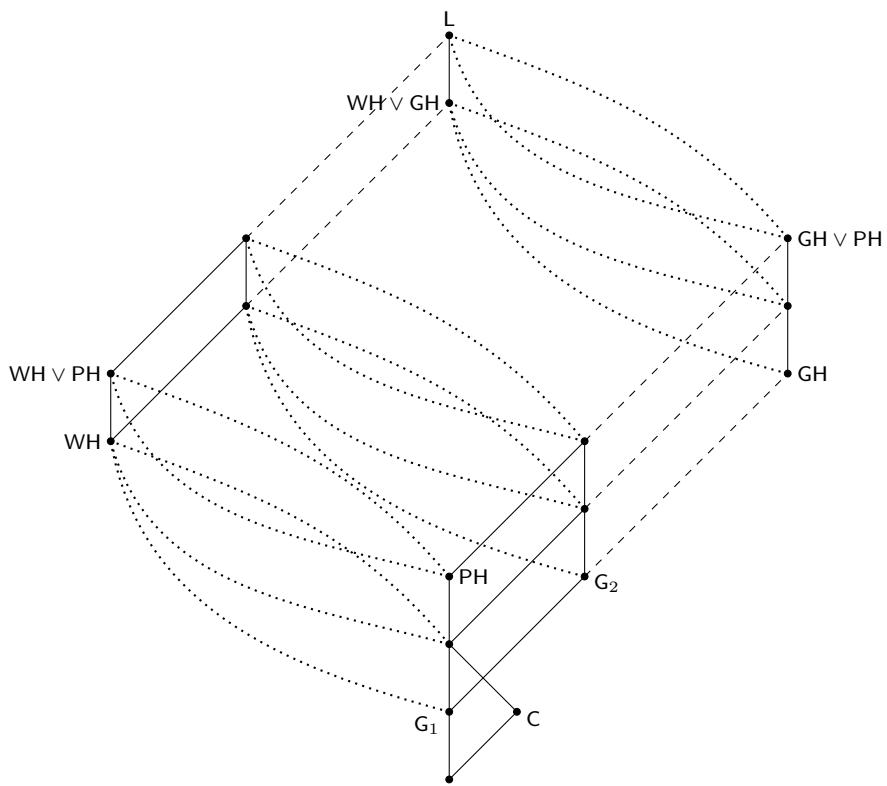

Clearly Figure \ref{l} is rather sketchy, since all the equal ovals are
isomorphic intervals and they are all connected.
However, using the same techniques as in \cite{AglianoPanti1999} it is not hard to prove the following:

\begin{lemma} Let $\vv V$ be a subvariety of $\vv L$ not above $\mathsf{WH}$ or $\mathsf{GH}$.
Then there are finite $I,J \sse \mathbb N$ (possibly empty) and an $n \in \mathbb N$ (possibly $0$) such that
$$
\vv V= \VV(\{\alg G_n,\two \oplus \alg C_\o,\alg {C}_i,\alg{C}^\infty_j: i
\in I, j\in J\})
$$
\end{lemma}

It follow that if $\vv V \sse \vv L$ and $\vv V$ is not above $\mathsf{WH}$ or $\mathsf{GH}$, then $\Lambda(\vv V)$ is finite,
so Corollary \ref{corbh} applies and then all the subvarieties generated by a chain of finite index are strictly join irreducible. Note that $\mathsf{GH}$ and $\mathsf{WH}$ are join irreducible but not strictly join irreducible
in $\Lambda(\vv V)$, since $\mathsf{WH}$ is generated by any Wajsberg chain of infinite rank and $\sf GH$ is generated by the chain
$[0,1]_\vv G$, where $xy=x\meet y$ in the usual ordering.

\section{Generalized rotations}

In this section we are going to use the results contained in \cite{AglianoUgolini2019a} to lift information about strictly join irreducible varieties of basic hoops to some subvarieties of $\mathsf{MTL}$-algebras. In particular, starting from a variety of basic hoops (or more in general, of representable $\mathsf{CIRL}$s), we can use their \emph{generalized rotation} (introduced in \cite{BMU2018}) to generate a variety of $\mathsf{MTL}$-algebras, and we are able to describe the relations between the respective lattices of subvarieties. Let us be more precise.

Consider a class of representable commutative residuated integral lattices ($\mathsf{RCIRL}$s for short) $\vv K$, a natural number $n \geq 2$, and a \textbf{term defined rotation}, i.e.: a unary term $\delta$ in the language of residuated lattices that on every algebra in $\alg A \in \vv K$ defines a lattice homomorphism that is also a nucleus (a closure operator satisfying $\delta(x) \cdot \delta(y) \leq \delta(x \cdot y$) for all $x, y \in A$).
It is easy to see that $\d(x) =x$ and $\d(x) =1$ are both term defined rotations on any class of $\mathsf{CIRL}$s. Then for every algebra $\alg A \in \vv K$ we can consider its \textbf{generalized ${\bf n}$-rotation}, whose definition we write here for the reader's convenience (for more details see \cite{BMU2018}).

The lattice structure is given by the disjoint union of $\alg A$ and $\delta[A]' = \{(\delta(a))': a \in A\}$ with the dualized order:  $$ a' < b \mbox{ and } a' \leq b' \mbox{ iff } b \leq a$$
to which we add a chain of $n-2$ elements strictly between $\alg A$ and $\delta[A]'$.
Intuitively, the resulting algebra has as skeleton a \L ukasiewicz chain of $n$ elements, {\bf \L}$_{n-1}$, where the top is $1$ and the bottom is $0 = 1'$. Let us call the elements of the \L ukasiewicz chain: $$0 = l_{0} < l_{1} < \ldots < l_{n-2} < l_{n-1} = 1.$$
$\alg A$ is a substructure, the products in $\delta[A]'$ are all defined to be the bottom element $0=1'$. Moreover $$a\cdot b' = (a \to b)', \quad a \to  b' = (\delta(b \cdot a))', \quad a' \to b' = b \to a$$
$$a l_{i} = l_{i} = l_{i}a, \quad b'l_{i} = 0 = l_{i}b' \quad ( \mbox{ for } i \notin\{0,n-1\}).$$
The obtained structure, that we shall denote with $\alg A^{\d_n}$, is an $\mathsf{MTL}$-algebra.
It is relevant to notice that the term defined rotation $\delta$ coincides exactly with the double negation of $\alg A^{\d_n}$ when restricted to $\alg A$. Moreover, $\alg A$ is the radical of $\alg A^{\d_{n}}$ (and its only maximal filter).

We can now define subvarieties of {\sf MTL} in the following way: let $\vv K$ be any class of representable $\mathsf{CIRL}$s, $n \geq 2$, $\d$ a term defined rotation, and we define
$$
\vv K^{\d_n} = \{\alg A^{\d_m}:m-1\mathrel{|}n-1,  \alg A \in \vv K \}.
$$
From now on we will write $\d$ for $\d_2$ (this coincides with the construction in \cite{AFU2017}). Whenever we consider a variety of $\mathsf{CIRL}$s $\vv V$, we will write $\vv V^{\delta_{n}}$ for the generated variety ${\bf V} (\vv V^{\delta_{n}})$. Moreover, we will write $\vv V_t$ for the class of totally ordered members in $\vv V$.

When we use $\delta$ as the identity map, we find interesting examples where the algebras obtained by rotations are involutive. We shall call this particular instance of the construction a \textbf{generalized disconnected $n$-rotation}. For instance, with $\vv V$ the variety of G\"odel hoops, its generalized disconnected $3$-rotation $\vv V^{\delta_{3}}$ is the variety of nilpotent minimum algebras $\mathsf{NM}$. Starting with a variety of basic hoops $\vv V$, the only disconnected $n$-rotation that is a subvariety of $\mathsf{BL}$-algebras is the variety generated by perfect MV-algebras, given by $2$-rotations of cancellative hoops (Proposition 3.4 \cite{AglianoUgolini2019a}). Thus with this particular construction we are able to move from basic hoops to $\mathsf{MTL}$-algebras. Notice that if  instead we use the term defined rotation $\d(x) =1$ on varieties of basic hoops, then  we obtain varieties of $\mathsf{BL}$-algebras. In particular, for $n = 2$, we get the variety of $\mathsf{SBL}$ of pseudocomplemented (or, equivalently, Stonean) $\mathsf{BL}$-algebras.
 Let us recall some results about the lattice of subvarieties of term-defined rotations obtained in  \cite{AglianoUgolini2019a}.

\begin{lemma}[\cite{AglianoUgolini2019a}]\label{lemma:HSProtation}
Let $\vv V$ be a class of $\mathsf{RCIRL}$s, $\delta$ a term-defined rotation and $n \geq 2$. Then the following are equivalent:
\begin{enumerate}
\item $\alg A \in \HH\SU\PP_u(\vv K)$
\item $\alg A^{\delta_{m}} \in \HH\SU\PP_u(\vv K^{\delta_{n}})$ for every $m: m-1 \mathrel{|} n-1$.
\end{enumerate}
\end{lemma}
In the case where $\vv K$ consists of a single algebra, we can actually obtain some extra information.
\begin{lemma}\label{lemma:HSPu2}
Let $\delta$ a term-defined rotation, $n \geq 2$ and $m-1 \mathrel{|} n-1$. Then $\alg D^{\delta_{m}} \in \HH\SU\PP_{u}(\alg B^{\delta_{n}})$ if and only if $\alg D^{\delta_{n}} \in \HH\SU\PP_{u}(\alg B^{\delta_{n}})$.
\end{lemma}
\begin{proof}
The right-to-left direction is a consequence of Lemma \ref{lemma:HSProtation}.
Now suppose $\alg D^{\delta_{m}} \in \HH\SU\PP_{u}(\alg B^{\delta_{n}})$ with $m-1 \mathrel{|} n-1$; we want to show that also $\alg D^{\delta_{n}} \in \HH\SU\PP_{u}(\alg B^{\delta_{n}})$. Notice that $\alg B^{\delta_{n}}$ is directly indecomposable (see \cite{BMU2018}), which is a first order property in ${\sf FL_{ew}}$-algebras, since it corresponds to the fact that the only complemented elements are $0$ and $1$, which can clearly be expressed with a first order sentence. Thus all ultrapowers of $\alg B^{\delta_{n}}$ are also directly indecomposable.
All directly indecomposable algebras in a variety generated by generalized rotations are generalized rotations themselves (Theorem 4.8, \cite{BMU2018}).
Moreover, having an MV-skeleton of $n$ elements can also be expressed as a first order property, using the term $\gamma_{n}$ defined in \cite{BMU2018}. In generalized rotations the term $\g_n$ is the identity map only on the elements of the MV-skeleton, evaluates to $1$ in the elements of the radical and evaluates to $0$ in the elements of the coradical. Therefore if one wishes to describe with a first order sentence that a generalized rotation has an MV-skeleton of $n$ elements, one can say that there exist $n$ elements $x_{1} \ldots x_{n}$, pairwise different, such that $\gamma_{n}(x_{i}) = x_{i}$ for $i = 1 \ldots k$, and that are the only ones with such property. So we can conclude that any ultrapower of $\alg B^{\delta_{n}}$ is a generalized rotation $\alg U^{\delta_{n}}$. Thus $\alg D^{\delta_{m}} \in \HH\SU(\alg U^{\delta_{n}})$ and of course $\alg D^{\delta_{m}} \in \HH\SU(\alg U^{\delta_{m}})$. Thus, via Lemma \ref{lemma:HSProtation}, $\alg D \in \HH\SU(\alg U)$ and then $\alg D^{\delta_{n}} \in \HH\SU(\alg U^{\delta_{n}})$; this  implies $\alg D^{\delta_{n}} \in \HH\SU\PP_{u}(\alg B^{\delta_{n}})$ and the proof is completed.
\end{proof}
The following is a consequence of Lemma \ref{lemma:HSProtation}.
\begin{proposition}\label{prop:sub}[\cite{AglianoUgolini2019a} ]
Let $\vv V, \vv W$ be subvarieties of $\mathsf{RCIRL}$, $\delta$ a term-defined rotation and $n \geq 2$.
Then $\vv V^{\delta_{m}} \subseteq \vv W^{\delta_{n}}$ iff $m-1 \mathrel{|} n-1$ and $\vv V\subseteq \vv W$.
\end{proposition}
As a special case notice that, given  a subvariety $\vv K$ of $\mathsf{RCIRL}$ and $\d$  a term defined rotation, the lattice of subvarieties of $\vv K$ and the lattice of subvarieties of $\vv K^\d$ are isomorphic. In the case of varieties obtained as rotations of basic hoops, we can also say something more.
\begin{lemma}\label{lemma:rotationcompletion}
For any subvariety $\vv V$ of $\sf BH$, $\Lambda(\vv V^{\delta_{n}})$ is the join dense completion of $\Phi_\vv V^{\delta_{n}} = \{{\bf V}(\alg A^{\delta_{m}}) : \alg A \in \mathcal{F}_{\vv V}, m-1 \mathrel{|} n-1 \}$.
\end{lemma}
\begin{proof} Let $\alg B$ be a totally ordered algebra in $\vv V^{\delta_{n}}$, then $\alg B$ is isomorphic to a generalized rotation $\alg A^{\delta_{m}}$ with $m-1 \mathrel{|} n-1$. Moreover, if $\alg B$ is finitely generated, then $\alg A \in \mathcal{F}_{\vv V}$, since $\alg B$ is completely determined by its MV-skeleton {\bf \L}$_{m-1}$ and the totally ordered basic hoop that is its radical $\alg A$, that cannot have infinitely many sum-irreducible components if it is finitely generated. Then the proof continues as in Lemma \ref{joindense}, and we can conclude that $\Lambda(\vv V^{\delta_{n}})$ is the join dense completion of $\Phi_\vv V^{\delta_{n}}$.
\end{proof}

\begin{lemma}\label{lemma:rotationjoin}
Let $\vv V, \vv W_{1}\ldots \vv W_{k}$ be subvarieties of $\mathsf{RCIRL}$, $\delta$ a term-defined rotation and $n \geq 2$. Then $$\vv V = \bigvee_{i \in I} \vv W_{i} \;\;\mbox{ if and only if }\;\;\vv V^{\delta_{n}} = \bigvee_{i \in I} \vv W_{i}^{\delta_{n}} $$
\end{lemma}
\begin{proof}
First, notice that given a $\mathsf{RCIRL}$ $\alg A $, $\alg A$ is subdirectly irreducible if and only if $\alg A^{\delta_{m}} $ is subdirectly irreducible. Indeed, in the generalized rotation construction $\alg A$ is a congruence filter of $\alg A^{\delta_{m}} $, thus the former has a unique atomic filter if and only if the latter does.

Let us show the left-to-right direction first, thus we suppose that $\vv V = \bigvee_{i \in I} \vv W_{i}$. Now, as a consequence of Lemma \ref{lemma:HSProtation}, $\vv V^{\d_n}$ is the variety generated by algebras in the set $\{\alg A^{\d_n}:  \alg A \in \vv V, \alg A$ subdirectly irreducible$\}$. Notice that in particular all the algebras $\alg A^{\d_{m}}$ for $m-1\mathrel{|}n-1$ are subalgebras of $\alg A^{\d_n}$.
Now we consider a subdirectly irreducible algebra $\alg A \in \vv V$, and we call $\mathsf{siW}_{i}^{\delta_{n}}$ and $\mathsf{siW}_{i}$ the class of subdirectly irreducible algebras in $\vv W_{i}^{\delta_{n}}$ and $\vv W_{i}$ respectively. We get:
$$
\begin{array}{rcl}
\alg A^{\d_{n}} \in \vv V^{\d_{n}} & \Leftrightarrow& \alg A^{\delta_{m}} \in \vv V^{\d_{n}}\mbox{ for all } m-1\mathrel{|}n-1\\
& \Leftrightarrow& \alg A \in \vv V\\
& \Leftrightarrow& \alg A \in  \displaystyle\bigvee_{i \in I} \vv W_{i}\\
& \Leftrightarrow&  \alg A \in {\bf HSP}_{u}(\displaystyle\bigcup_{i \in I}\mathsf{siW}_{i})\\
& \Leftrightarrow&  \alg A ^{\delta_{m}}\in {\bf HSP}_{u}(\displaystyle\bigcup_{i \in I} \mathsf{siW}_{i}^{\delta_{n}}) \mbox{ for all } m-1\mathrel{|}n-1\\
& \Leftrightarrow&  \alg A ^{\delta_{m}}\in \displaystyle\bigvee_{i \in I} \vv W_{i}^{\delta_{n}} \mbox{ for all } m-1\mathrel{|}n-1\\
& \Leftrightarrow& \alg A ^{\delta_{n}}\in \displaystyle\bigvee_{i \in I} \vv W_{i}^{\delta_{n}}.
\end{array}
$$
The first and last equivalence follow from the fact that $\alg A^{\delta_{m}}$ is a subalgebra of $\alg A^{\delta_{n}}$ iff $m-1 \mathrel{|} n-1$; the second and the third to last equivalences follow from Lemma \ref{lemma:HSProtation}; the third one follows by hypothesis. So it remains to show  the fourth equivalence (the proof of the second to last one being analogous).
Since $\alg A$ is subdirectly irreducible and  residuated lattices are congruence distributive, by J\`onnson's Lemma $\alg A \in {\bf V}(\vv K)$ iff $\alg A \in {\bf HSP}_{u}(\vv K)$ for any class of algebras $\vv K$. The equivalence then follows from the fact that $\bigcup_{i \in I} \mathsf{siW}_{i}$ is a set of generators for $\bigvee_{i \in I} \vv W_{i}$. Indeed, clearly ${\bf V}(\bigcup_{i \in I} \mathsf{siW}_{i}) \subseteq \bigvee_{i \in I}\vv W_{i}$. Moreover, for all $i \in I$, $\vv W_{i} \subseteq {\bf V}(\bigcup_{i = 1 \ldots k} \mathsf{siW}_{i})$ thus ${\bf V}(\bigcup_{i \in I} \mathsf{siW}_{i}) = \bigvee_{i \in I} \vv W_{i}$.

We have showed that $\vv V^{\delta_{m}} = \bigvee_{i \in I }\vv W_{i}^{\delta_{n}}$.
The proof of the right-to-left direction uses the very same equivalences and is left to the reader.
\end{proof}
Now we can prove the following.

\begin{proposition}\label{prop:sjirotation}
Let $\vv V \subseteq \vv W$ be subvarieties of $\mathsf{BH}$, $\d$ a term defined rotation and $m,n \geq 2$. Then $\vv V^{\delta_{m}}$ is strictly join irreducible in the lattice of subvarieties of $\vv W^{\delta_{n}}$ iff $\vv V$ is strictly join irreducible in the lattice of subvarieties of $\vv W$ and $m-1 \mathrel{|} n-1.$
\end{proposition}
\begin{proof}
The left-to-right direction is relatively easy to see by contraposition. Indeed, first of all if $m-1$ does not divide $n-1$ then $\vv V^{\delta_{m}}$ is not a subvariety of $\vv W^{\delta_{n}}$ by Proposition \ref{prop:sub}. Moreover, if $\vv V$ is not strictly join irreducible, then $\vv V = \bigvee_{i \in I} \vv W_{i}$ for some proper subvarieties $\vv W_{i}, i \in I$, but then $\vv V^{\delta_{m}} = \bigvee_{i \in I} \vv W_{i}^{\delta_{m}}$ by Lemma \ref{lemma:rotationjoin}. Next $\vv W_{i}^{\delta_{m}}$ is a subvariety of $\vv V^{\delta_{m}}$ by Proposition \ref{prop:sub}.
It is also proper subvariety: let $\alg A \in \vv V, \alg A \notin \vv W_{i}$. Then by Lemmas \ref{lemma:HSProtation} and \ref{lemma:HSPu2}, $\alg A^{\delta_{m}} \in \vv V^{\d_{m}}$ but $\alg A^{\delta_{m}} \notin \vv W_{i}^{\delta_{m}}$. Thus $\vv V^{\delta_{m}}$ is not strictly join irreducible.

Suppose now that $\vv V$ is strictly join irreducible and $m-1 \mathrel{|} n-1$. Then $\vv V = {\bf V}(\alg A)$ for some chain $\alg A$ of finite order whose Wajsberg components are either cancellative or are bounded and have finite rank (Theorem \ref{thm:sjibasic}). We can show that $\vv V^{\delta_{m}} = {\bf V}( \alg A^{\delta_{m}})$. One inclusion is obvious, while for the other one we use the fact that subdirectly irreducible algebras in $\vv V^{\delta_{m}}$ are generalized rotations of subdirectly irreducible algebras in $\vv V$. So let's take a subdirectly irreducible algebra $\alg C^{\delta_{k}} \in \vv V^{\delta_{m}}$, with $(k-1) \mathrel{|} (m-1)$. Then $\alg C$ is a subdirectly irreducible algebra in $\vv V = {\bf V}(\alg A)$, which implies that $\alg C \in \HH\SU\PP_{u}(\alg A)$. Via Lemma \ref{lemma:HSProtation}, this implies that $\alg C^{\delta_{k}} \in \HH\SU\PP_{u}(\alg A^{\delta_{m}})$, and thus every subdirectly irreducible in $\vv V^{\delta_{m}}$ belongs to ${\bf V}( \alg A^{\delta_{m}})$, which completes the proof of the equality  $\vv V^{\delta_{m}} = {\bf V}( \alg A^{\delta_{m}})$.

We have shown that $\vv V^{\delta_{m}}$ is generated by a totally ordered algebra, which by Corollary \ref{prelinear} is also well-connected; by applying Theorem \ref{mainwell} we get that $\vv V^{\delta_{m}}$ is join irreducible. We now show that $\Lambda(\vv V^{\delta_{m}})$ is finite, which will settle the proof. In particular, given Lemma \ref{lemma:rotationcompletion}, it suffices to show that $\Phi_{\vv V}^{\delta_{n}}$ is finite. In order to prove this, we will show that given $\alg B$ and $\alg C$ in $ \mathcal{F}_{\vv V}$, for all $k-1 \mathrel{|} m-1$ we get
$$
\mbox{ if }\quad \HH\SU\PP_{u}(\alg B) = \HH\SU\PP_{u}(\alg C)\quad \mbox{ then }\quad  \HH\SU\PP_{u}(\alg B^{\delta_{k}}) = \HH\SU\PP_{u}(\alg C^{\delta_{k}}).
$$
Since $\vv V$ is strictly join irreducible, $\Phi_{\vv V}$ is finite as shown in the proof of Theorem \ref{thm:sjibasic}, which then implies that also $\Phi_{\vv V}^{\delta_{n}}$ is finite.

So let now $\alg B$ and $\alg C$ in $\mathcal{F}_{\vv V}$ such that $\HH\SU\PP_{u}(\alg B) = \HH\SU\PP_{u}(\alg C)$. We show that for any $k$ such that  $k-1 \mathrel{|} m-1$, $\HH\SU\PP_{u}(\alg B^{\delta_{k}}) = \HH\SU\PP_{u}(\alg C^{\delta_{k}}).$

Indeed consider an algebra in $\HH\SU\PP_{u}(\alg B^{\delta_{k}})$; then it is equal to $\alg D^{\delta_{j}}$ for some totally ordered basic hoop $\alg D$ and $j-1 \mathrel{|} k-1$. Thus by Lemma \ref{lemma:HSPu2}, $\alg D^{\delta_{k}} \in \HH\SU\PP_{u}(\alg B^{\delta_{k}})$, and then $\alg D^{\delta_{l}} \in \HH\SU\PP_{u}(\alg B^{\delta_{k}})$ for all $l -1 \mathrel{|} k-1$ (since they are subalgebras of $\alg D^{\delta_{k}}$); hence by Lemma \ref{lemma:HSProtation}, $\alg D \in  \HH\SU\PP_{u}(\alg B) = \HH\SU\PP_{u}(\alg C)$ by hypothesis, and then (using again Lemma \ref{lemma:HSProtation}), $\alg D^{\delta_{j}} \in \HH\SU\PP_{u}(\alg C^{\delta_{k}})$, which implies that $\HH\SU\PP_{u}(\alg B^{\delta_{k}}) \subseteq \HH\SU\PP_{u}(\alg C^{\delta_{k}})$. The equality follows from the fact that we can use the same reasoning starting with an algebra in $\HH\SU\PP_{u}(\alg C^{\delta_{k}})$. Therefore  the proof is completed.
\end{proof}

Moreover, we can also use the following result to obtain information on splitting algebras in varieties generated by generalized rotations. We recall that an algebra $\alg A \in \vv V$ is splitting in the lattice of subvarieties of $\vv V$ if there is a subvariety $\vv W_{\alg A}$ of $\vv V$ (the conjugate variety of $\alg A$) such that for any variety $\vv U \subseteq \vv V$ either $\vv U \subseteq \vv W_{\alg A}$ or $\alg A \in \vv U$.
\begin{proposition}\label{prop:sjisplitting}
Let $\alg A$ be a totally ordered basic hoop of finite index such that its Wajsberg components are either cancellative or bounded with finite rank, and let $\vv V = {\bf V}(\alg A)$. Then for all suitable term defined rotations $\d$, $\alg A^{\d_m}$ is splitting in $\vv V^{\d_n}$ for any $m$ such that $m-1 \mathrel{|} n-1$.
\end{proposition}
\begin{proof}
It follows from Theorem \ref{thm:sjibasic} that $\vv V$ is strictly join irreducible in $\mathsf{BH}$, then the result follows by applying Lemma 3.3 in \cite{AglianoUgolini2019a}.
\end{proof}

\subsection{Some examples}
We shall now give some specific applications of the theorems of this section. For details about the subvarieties mentioned here and their axiomatization, we refer the reader to \cite{AglianoUgolini2019a}.
Let us start considering the generalized rotations obtained with the term defined rotation $\delta(x) =\ell(x) = 1$. Then $\BH^{\ell_n}$ is the subvariety of $\BL$ generated by all the so-called $n$-liftings of basic hoops: ordinal sums of the kind {\bf \L}$_{n} \oplus \alg B$, for a basic hoop $\alg B$. If $t_{n}(x) \approx 1$ is the equation axiomatizing $\mathsf{V}(${\bf \L}$_{n})$, then the variety $\BH^{\ell_n}$ is axiomatized by $t_{n}(\neg\neg x) \approx 1$. Clearly, for $n = 2$ we have the variety of Stonean $\mathsf{BL}$-algebras.
Then Proposition \ref{prop:sjirotation} applies to this variety, and we can say that if $\vv V$ is a variety of basic hoops, then $\vv V^{\ell_{m}}$ is strictly join irreducible in the lattice of subvarieties of $\BH^{\ell_n}$ iff $\vv V$ is strictly join irreducible in the lattice of subvarieties of basic hoops and $m-1 \mathrel{|} n-1.$

More in general, the following characterization is a direct consequence of Proposition \ref{prop:sjirotation} and Theorem \ref{thm:sjibasic}.
\begin{corollary}
Let $\vv V$ be a subvariety of $\BH$, $\d$ a term-defined rotation and $m, n \geq 2$. $\vv V^{\d_{m}}$ is strictly join irreducible in the lattice of subvarieties of $\BH^{\d_n}$ if and only if $m-1 \mathrel{|} n-1$ and $ \vv V = \VV(\alg A)$ for some $\alg A \in \mathcal F_\vv V$ such that its Wajsberg components are either cancellative or are bounded and have finite rank.
\end{corollary}
For the rest of the section we are going to focus on term defined rotation the identity map, $\delta = id$. Thus the resulting varieties will be involutive subvarieties of $\mathsf{MTL}$.

\subsubsection{Nilpotent minimum varieties}
Let us first consider involutive varieties generated by generalized disconnected rotations of G\"odel hoops, that we called in \cite{AglianoUgolini2019a} \emph{nilpotent minimum varieties}, in similarity to the specific case of nilpotent minimum algebras $\mathsf{NM} = \mathsf{GH}^{\d_{3}}$.

Every proper subvariety $\vv G_{k} = {\bf V}(\alg G_{k})$ of G\"odel hoops is strictly join irreducible. Thus $\alg G_{k}^{\d_{m}}$ is strictly join irreducible in the lattice of subvarieties of $\vv G^{\d_{n}}$ whenever $m-1 \mathrel{|} n-1$.

Moreover, all algebras $\alg G_{k}$ satisfy the hypothesis of Proposition \ref{prop:sjirotation}, thus all rotations $\alg G_{k}^{\d_{m}}$ are splitting in $G_{k}^{\d_{n}}$ whenever $m-1 \mathrel{|} n-1$.

\subsubsection{Nilpotent \L ukasiewicz varieties}
We call \emph{nilpotent \L ukasiewicz varieties} all subvarieties of $\BH^{\d_n}$ generated by disconnected $n$-rotations of Wajsberg hoops. $\mathsf{WH}^{\d_{n}}$ can be axiomatized by
\begin{equation}
(\nabla_n(x) \land \nabla_n(y)) \to (((x \imp y) \imp y) \to ((y \imp x) \imp x)) \app 1.
\end{equation}
Proper subvarieties of $\WH$ are all generated by finitely many chains \cite{AglianoPanti1999} so their lattice of subvarieties is finite. Any proper variety of Wajsberg hoops $\vv V$ is axiomatized (modulo basic hoops) by a single equation in one variable of the form $t_\vv V (x) \app 1$, and $\vv V^{\d_n}$ is axiomatized by $\neg x^n \join t_\vv V(x)\app 1$ (modulo  $\BH^{\d_n}$).

Now, the proper subvarieties of $\WH$ that are strictly join irreducible are the varieties generated by a single Wajsberg chain that is either cancellative or with finite rank, i.e., either $\alg{\L}_n$ or $\alg{\L}_n^\infty$ for some $n$. Calling $\mathsf{W}_{n}$ and $\mathsf{W}_{n}^{\infty}$ the varieties generated by, respectively, $\alg{\L}_n$ or $\alg{\L}_n^\infty$, for all $m \geq 2$, $\mathsf{W}_{n}^{\d_{m}}$ and $(\mathsf{W}_{n}^{\infty})^{\d_{m}}$ are strictly join irreducible in $\mathsf{W}^{\d_{n}}$ whenever $m-1 \mathrel{|} n-1$.

\subsubsection{Nilpotent product varieties}
We refer to the varieties $\vv C^{\d_n}$ for $n \ge 2$ as {\bf nilpotent product varieties}.  Since cancellative hoops are axiomatized relative to Wajsberg hoop by $(x \imp x^2)\imp x \app 1$ the variety $\vv C^{\d_n}$ is axiomatized by
$$
\neg x^n \imp ((x \imp x^2) \imp x) \app 1.
$$
The lattice of subvarieties can be easily described knowing that: $\vv C^{\d_n} = \VV(\alg C_\o^{\d_n})$.
For each $\vv C^{\d_n}$, the varieties  $\vv C^{\d_m}$ where $m-1 \mathrel{|} n-1$ are strictly join irreducible.

\section{Linear varieties}

A variety $\vv V$ of residuated lattices is {\bf linear} if $\Lambda(\vv V)$ is totally ordered. Clearly each subvariety of a linear variety is join irreducible; via Theorem \ref{mainwell} each subvariety is then generated by a single $\Gamma^{n}$-connected algebra. Since obviously if every subvariety of a variety is join irreducible, then the variety is linear, we have:

\begin{lemma} A variety $\vv V$ is linear if and only if each subvariety of $\vv V$  is generated by a single  $\Gamma^{n}$-connected algebra. Moreover if the order type of the chain is $\o+1$, then every subvariety is generated by a single subdirectly irreducible algebra.
\end{lemma}

If the variety is also representable, then each subvariety is generated by a single chain. The prototype of a variety satisfying this is the variety $\mathsf{GH}$ of G\"odel hoops; $\Lambda(\sf GH)$ is the chain $\vv G_1 \sse \vv G_2 \sse \dots \vv G$ and each $\vv G_n$ is generated by $\bigoplus_{i=1}^n \mathbf 2$ which is of course subdirectly irreducible. Therefore $\sf GH$ is generated by any infinite set of chain of finite index or else by a single chain of infinite index. This is no coincidence. First we need a simple lemma:

\begin{lemma}\label{finiteq} Let $\alg A$ be totally ordered residuated lattice. Then $\alg A$ is finite if and only if $\alg A \vDash \bigvee_{i<n} x_{i+1}\lr x_i \geq 1$ for some $n$.
\end{lemma}
\begin{proof}Let $\alg A$ be finite, say $|A|=n - 1$. Therefore if $\vuc an \in A$ there must be some $i<n$ with $a_i \leq a_{i+1}$; this implies $\bigvee_{i<n} a_{i+1}\lr a_i \geq 1$ and so $\alg A$ satisfies the desired equation.

Conversely suppose that $\alg A$ is infinite; then for any $n$ we can find $\vuc an \in A$ with $a_{n} < \dots < a_1$, so that $a_{i+1}\lr a_i < 1$ for all $i<n$.  As $\alg A$ is a chain $1$ is join prime,
so $\bigvee_{i<n} a_{i+1}\lr a_i < 1$. Therefore $\alg A$ does not satisfy any of the desired equations.
\end{proof}

\begin{lemma}\label{techlemma4} Let $\vv V$ be any linear variety of representable residuated lattices.
\begin{enumerate}
\item Given any infinite chain $\alg A$, $\VV(\alg A)$ contains the variety  $\vv V_c$ generated by all chains in $\vv V$.
\item If there is an infinite chain in $\vv V$, $\vv V$ has the finite model property (FMP for short) if and only if $\vv V_{c} = \vv V(\alg A)$ for all infinite chains $\alg A \in \vv V$.
\item For any fixed $n$, all totally ordered members of $\vv V$ of cardinality $n$ are isomorphic.
\end{enumerate}
\end{lemma}
\begin{proof}  Let us prove (1) first. Let $\alg A$ be any infinite chain in $\vv V$ and let $\alg B $ be a finite chain in $\vv V$. Then since $\vv V$ is linear, either $\VV(\alg A) \sse \VV(\alg B)$ or $\VV(\alg B) \sse \VV(\alg A)$. But since $\alg B$ is finite, by Lemma \ref{finiteq} it satisfies an identity that $\alg A$ does not, thus $\VV(\alg A) \not\sse \VV(\alg B)$ and then $\VV(\alg B) \sse \VV(\alg A)$. Thus $\vv V(\alg A)$ contains all finite chains in $\vv V$, which proves (1).

We now show (2). We have already shown that $\vv V_{c} \subseteq V(\alg A)$ for all infinite chains $\alg A \in \vv V$. We recall that a variety has the finite model property iff it is generated by its finite members. So the FMP holds for $\vv V$, then $\vv V_{c} = \vv V$ and then $\vv V_{c} = \vv V(\alg A) = \vv V$. Vice versa, suppose $\vv V_{c} = \vv V(\alg A)$ for all infinite chains $\alg A \in \vv V$. If an equation fails in $\vv V$, it fails in some chain, possibly infinite, $\alg B$. But $\vv V(\alg B) = \vv V_{c}$, thus the equation must fail in a finite chain. This means that the FMP holds for $\vv V$.

Finally we prove (3). Let $\alg A, \alg B \in \vv V$ two chains having the same cardinality with $\alg A \not \cong \alg B$; as they are finite and totally ordered they are subdirectly irreducible and J\'onnson's Lemma
implies that $\alg B \notin \VV(\alg A)$ and $\alg A \notin \VV(\alg B)$. Since $\vv V$ is linear this is a contradiction and (2) holds.
\end{proof}

\begin{corollary} \label{cortech4} Let $\vv V$ be a linear variety of representable residuated lattices with the  finite model property (i.e. $\vv V$ is generated by its finite algebras).
\begin{enumerate}
\item Every infinite chain generates $\vv V$, and the only proper subvarieties of $\vv V$ are generated by a single finite chain (up to isomorphism).
\item If $\vv V$ is generated by a finite chain, then $\Lambda(\vv V)$ is finite.
\item If $\vv V$ is generated by an infinite chain, $\Lambda(\vv V)$ has order type $\o+1$.
\end{enumerate}
\end{corollary}
\begin{proof}  Since $\vv V$ has the finite model property, by Lemma \ref{techlemma4}, if $\vv V$ has an infinite chain it generates the whole variety since it generates all finite chains. Therefore all the proper subvarieties must be generated by a finite chain. If $\vv V$ is generated by a finite chain $\alg A$ then by J\'onnson's Lemma any totally ordered member of $\vv V$ must be in $\HH\SU(\alg A)$; hence they are all finite and (1) holds.

If $\vv V$ is generated by a finite chain of cardinality $n$, then in $\vv V $ there are only finite chains of cardinality at most $n$ (by Lemma \ref{finiteq}). Since finite chain of equal cardinality are isomorphic (by Lemma \ref{techlemma4}(3)) $\Lambda(\vv V)$ must be finite.

Finally, if $\vv V$ is generated by an infinite chain, then the only proper subvarieties are generated by finite chains; by Lemma \ref{techlemma4}(3) we can pick a representative chain $\alg A_{n+1}$ with $n$-elements for any $n \in \mathbb N$ and $\Lambda(\vv V)$ is
$$
\VV(\alg A_0) < \VV(\alg A_1) < \dots < \VV(\alg A_{n+1}) <\dots < \vv V.
$$
Therefore (3) holds.
\end{proof}

Note that the proofs of Lemma \ref{techlemma4} and Corollary \ref{cortech4} depend basically on the fact that we can distinguish equationally finite chains of different cardinalities and finite chains from infinite chains.

If $\vv V$ is a subvariety of $\mathsf{RPsH}$ we can rephrase the previous results substituting the finiteness of the chain with the finiteness of the index. In other words:
\begin{corollary} Let $\vv V$ be a linear variety contained in $\mathsf{RPsH}$ with the finite model property.
\begin{enumerate}
\item $\vv V$ is generated by each chain of infinite index $\alg A$ (if any); hence
the only proper subvarieties of $\vv V$ are those generated by chains of finite index.
\item $\Lambda(\vv V)$ is either finite or it is a chain of type $\o +1$.
\end{enumerate}
\end{corollary}

Now we can use the classification of chains of finite index in $\mathsf{BH}$ to describe all linear varieties of basic hoops.  First we observe that the linear varieties of Wajsberg hoops are easy to classify simply by inspection
using the description of $\Lambda(\sf WH)$ in \cite{AglianoPanti1999}. We recall that we are denoting by $\vv W_n$ the variety generated the Wajsberg hoop with $n+1$-elements, {\bf \L}$_{n}$.

\begin{theorem}\label{linearwa}
 The only linear varieties of Wajsberg hoops are the variety $\vv C$ of cancellative hoops and $\vv W_n$ where either $n=1$ or $n$ is a prime power.
\end{theorem}

Observe that if $\alg A =\bigoplus_{i \in I}\alg A_i$ is a totally ordered hoop , then $\VV(\alg A_i) \sse \VV(\alg A)$ for all $i \in I$. Hence if $\VV(\alg A)$ happens to be linear, then by Theorem \ref{linearwa}, each $\alg A_i$ must be either $\alg C_\o$ or $\alg {L}_n$ where $n$ is a prime power.

\begin{lemma} \label{omegatwo} Let $\alg A =\bigoplus_{i\in I} \alg A_i$, with $|I|>1$,  be a totally ordered hoop such that $\VV(\alg A)$ is linear. Then
\begin{enumerate}
\item either $\alg A_i = \mathbf 2$ for $i\in I$
\item or $\alg A_i = \alg C_\o$ for $i\in I$.
\end{enumerate}
\end{lemma}
\begin{proof}
It follows from Theorem \ref{linearwa} that the only possible components are either $\alg C_{\o}$, $\alg 2$ or $\alg {L}_n$ for some prime power $n$.
Note that since $\VV(\mathbf 2)$ and $\VV(\alg C_\o)$ are incomparable varieties of Wajsberg hoops, $\mathbf 2$ and $\alg C_\o$ cannot appear both in the decomposition of $\alg A$. Since $\alg 2$ is a subalgebra of each $\alg {Wa}_n$, also $\alg {Wa}_n$ (for any $n$) and $\alg C_\o$ cannot appear both in the decomposition of $\alg A$. Thus if there is $i$ such that $\alg A_i = \alg C_\o$, then $\alg A_i = \alg C_\o$ for all $i\in I$ and (2) holds.

Suppose now that $\alg C_\o$ is not in the decomposition of $\alg A$, and suppose that for some $i$, $\alg A_i = \mathbf 2$. If there is a $j\ne i$ such that $\alg A_j = \alg {\L}_n$ where $n$ is a prime power, $n >1$, it follows immediately that $\VV(\alg {\L}_n)$ and $\VV(\mathbf 2 \oplus \mathbf 2)$ are incomparable subvarieties of $\VV(\mathbf 2 \oplus \alg{\L}_n)$ so $\VV(\alg A)$ is not linear. The same holds if there are two components of the kind $\alg {\L}_n$ and $\alg {\L}_m$, with $m, n$ (possibly equal) prime powers. Thus the only other possibility is that (1) holds.
\end{proof}

Now we can characterize all the linear varieties of basic hoops. Let us call $\Omega(\vv C)$ the variety generated by arbitrary ordinal sums of $\alg C_{\o}$.

\begin{theorem}\label{linhoops} A variety $\vv V$ of basic hoops is linear if and only if:
\begin{enumerate}
\item $\vv V$ is a linear variety of Wajsberg hoops or
\item $\vv V$ is a subvariety of $\vv G$ , or
\item $\vv V$ is a subvariety of the variety $\Omega(\vv C)$ of hoops satisfying $x \imp x^2 \le x$.
\end{enumerate}
\end{theorem}
\begin{proof}  We have already seen that  $\vv G$ is linear.  For the variety $\Omega(\vv C)$ we simply note that any chain of finite index has only cancellative hoops as components, since the equation   $x \imp x^2 \le x$, in conjunction with Tanaka's equation, implies cancellativity. Therefore, modulo $\equiv$, any chain of index $n$ is of the form $\bigoplus_{i<n} \alg C_\o$; it is obvious now that $\Omega(\vv C)$ is linear and generated by the chain $\bigoplus_{n\in \o} \alg C_\o$.

Conversely assume that $\vv V$ is a linear variety of basic hoops that does not consists entirely of Wajsberg hoops. Then $\vv V$ is generated by a chain $\alg A$ of index $>1$ (possibly infinite).  Hence by Lemma \ref{omegatwo} the components of $\alg A'$ are all equal to $\mathbf 2$ or to $\alg C_\o$ and the conclusion follows.
\end{proof}

Moreover, similarly to the case of Wajsberg hoops the varieties of Wajsberg algebras that are linear can be found by simply inspecting the description of the lattice of subvarieties in \cite{DiNolaLettieri1999}.
Let us call $\mathsf{Wa}_n$ the variety generated by the $n+1$ elements MV-chain.
\begin{lemma}\label{Wajsberglin} Let $\vv V$ be a variety of Wajsberg algebras; then $\vv V$ is linear if and only if it is either $\mathsf{Wa}_n$ where either $n=1$ or $n$ is a prime power, or else it is generated by the Chang algebra  (i.e. the zero-bounded version of $\alg{\L}_2^\infty$).
\end{lemma}

We can now apply our results to varieties obtained by generalized $n$-rotations. In particular, if we pick $\delta$ to be any term-defined rotation on $\mathsf{BH}$ and $n = 2$,
we can use the lattice isomorphism between $\Lambda(\mathsf{BH})$ and $\Lambda(\mathsf{BH}^{\delta_{2}})$ described after Proposition \ref{prop:sub}.
In particular the following result applies if $\delta = 1$ or $\delta= id$, that is to say, to $\mathsf{SBL}$ and the variety generated by disconnected rotations of basic hoops that we shall call $\mathsf{bIDL}$.

\begin{theorem}\label{linsbl} Let $\d$ be a term defined rotation, then a variety $\vv V$ contained in $\mathsf{BH}^{\delta_{2}}$ is linear if and only if $\vv V = \vv W^{\delta_{2}}$ where $\vv W$ is a linear variety of basic hoops.
\end{theorem}
Thus in particular  a subvariety of $\mathsf{SBL}$ is linear if and only if it is either
\begin{enumerate}
\ib a subvariety of the variety of G\"odel  algebras (bounded G\"odel hoops), or
\ib the variety generated by $\mathbf 2 \oplus \alg {Wa}_n$ where $n$ is a prime power, or
\ib the variety generated by $\mathbf 2 \oplus \alg A$, where $\alg A$ is a totally ordered member of $\Omega(\vv C)$.
\end{enumerate}
What about $\mathsf{BH}^{\d_n}$ for $n>2$? It turns out that the only linear varieties are those already described in Theorem \ref{linsbl}.

\begin{theorem}  Let $\d$ be a term defined rotation, then a variety $\vv V$ contained in $\mathsf{BH}^{\delta_{n}}$ is linear if and only if either
\begin{enumerate}
\item $\vv V=\mathsf{Wa}_n$ with $n=1$ or $n$ a prime power, or
\item $\vv V = \vv W^{\delta_{2}}$ where $\vv W$ is a linear variety of basic hoops.
\end{enumerate}
\end{theorem}
\begin{proof}
The right-to-left direction is easy to see, and is a consequence of Theorem \ref{linsbl} and Proposition \ref{prop:sub}. Suppose now that $\vv V$ contained in $\mathsf{BH}^{\delta_{n}}$ is linear, and it is not a linear variety of MV-algebras. Since it is linear it is also join irreducible, which means it is generated by a single subdirectly irreducible algebra (Theorem \ref{main3}), say $\alg A^{\delta_{m}} \neq \alg \L_{m-1}$, with $m-1 \mathrel{|} n-1$. If $m > 2$, then both the MV-skeleton ${\alg \L}_{m-1}$ and $\alg A^{\delta_{2}}$ are subalgebras of $\alg A^{\delta_{m}}$, and they generate incomparable subvarieties of $\vv V$. Thus necessarily $m = 2$. Thus the result follows from Theorem \ref{linsbl}.
\end{proof}

Finally  we show how we can get the classification of all linear varieties of $\mathsf{BL}$-algebras (already obtained in \cite{AguzzoliBianchi2019}) using our techniques.

\begin{theorem}\label{linearBL} A variety $\vv V$ of $\mathsf{BL}$-algebras is linear if and only if:
\begin{enumerate}
\item $\vv V$ is a linear variety of Wajsberg algebras, or
\item $\vv V$ is a linear variety contained in $\mathsf{SBL}$.
\end{enumerate}
\end{theorem}
\begin{proof}   Let $\alg A = \bigoplus_{i \in I}\alg A_i$ be a $\mathsf{BL}$-algebra such that $\VV(\alg A)$ is linear; if $\alg A$ is not a Wajsberg algebra, then $|I|>1$. Take $k \in I$, $k \ne 0$; then $\VV(\alg A_0 \oplus \alg A_k)$ must be a linear subvariety of $\VV(\alg A)$.  If  $\alg A_0 \ne \mathbf 2$ then $\alg A_0 \notin \HH\SU\PP_u(\mathbf 2)$,  hence $\mathbf 2 \oplus \alg A_k$ and $\alg A_0$ generate incomparable subvarieties of $\VV(\alg A_0 \oplus \alg A_k)$. But this is a contradiction so $\alg A_0 = \mathbf 2$ and $\alg A \in \mathsf{SBL}$.
\end{proof}
\section{Conclusions}
We have investigated in a very general setting the properties of being, respectively, join irreducible and strictly join irreducible in subvarieties of residuated lattices. Moreover, we have applied the results found to representable varieties, and obtained a precise characterization of (strictly) join irreducible varieties and linear varieties in the case of basic hoops. Finally, we made use of the generalized rotation construction to lift some of these results to subvarieties of $\mathsf{MTL}$-algebras.

As for future work, it is worth mentioning that a study of the techniques employed in  \cite{DvurecenskijHolland2007} and \cite{DvurecenskijHolland2009} could lead to an analogue characterization for varieties of representable pseudohoops.

Finally, we believe that the results and techniques developed in the first sections of this work can pave the way for many other possible applications than the ones considered here.

\section{Funding}
This work was supported by the European Union's Horizon 2020 research and innovation programme with a Marie Sk\l odowska-Curie grant [890616 to S.U.].

\providecommand{\bysame}{\leavevmode\hbox to3em{\hrulefill}\thinspace}
\providecommand{\MR}{\relax\ifhmode\unskip\space\fi MR }
\providecommand{\MRhref}[2]{%
  \href{http://www.ams.org/mathscinet-getitem?mr=#1}{#2}
}
\providecommand{\href}[2]{#2}

\end{document}